\documentclass[a4paper,11pt]{article}

\usepackage{fullpage}

\usepackage{amsbsy}
\usepackage{latexsym}
\usepackage{amsfonts}
\usepackage{amssymb}
\usepackage[usenames]{color}
\usepackage{amsmath,amsthm}

\let\oldproofname=\proofname
\renewcommand{\proofname}{\upshape\bfseries{\oldproofname}}

\DeclareMathOperator*{\essinf}{ess\,inf}
\DeclareMathOperator{\supp}{supp}

\providecommand{\abs}[1]{\lvert#1\rvert}

\providecommand{\norm}[1]{\lVert#1\rVert}

\providecommand{\Bignorm}[1]{\Bigl\lVert#1\Bigr\rVert}

\providecommand{\ceil}[1]{\lceil#1\rceil}

\newcommand{\be}{\begin{equation}}
\newcommand{\ee}{\end{equation}}
\newcommand{\iref}[1]{{\rm (\ref{#1})}}

 \newcommand\e{\varepsilon}
 
 \newcommand\nl{\newline}

\newtheorem{theorem}{Theorem}[section]
\newtheorem{lemma}[theorem]{Lemma}
\newtheorem{prop}[theorem]{Proposition}
\newtheorem{cor}[theorem]{Corollary}
\newtheorem{ass}[theorem]{Assumptions}

\theoremstyle{definition}

\newtheorem{remark}[theorem]{Remark}

\numberwithin{equation}{section}

\newcommand{\cF}{{\mathcal{F}}}

\newcommand{\cT}{\mathcal{T}}

\newcommand{\cS}{\mathcal{S}}
\newcommand{\cV}{\mathcal{V}}

\newcommand{\bn}{{\bf n}}
\newcommand{\bc}{{\bf c}}

\newcommand{\Chi}{\raise .3ex
\hbox{\large $\chi$}} 

\newcommand{\vp}{\varphi}

\newcommand{\R}{\mathbb{R}}

\newcommand{\E}{\mathbb{E}}

\newcommand{\N}{\mathbb{N}}

\begin{document}

\title
{Fully discrete approximation \\ of parametric and stochastic elliptic PDEs
\thanks{%
Markus Bachmayr acknowledges support by the Hausdorff Center of Mathematics, University of Bonn.
Albert Cohen is supported by the Institut Universitaire de France and
by the European Research Council under grant ERC AdG 338977 BREAD.
Dinh D\~ung's research work 
is funded by the Vietnam National Foundation for Science and Technology Development (NAFOSTED) under  Grant No.\ 102.01-2017.05.
Christoph Schwab is supported in part by the Swiss National Science Foundation.}
}
\author{ 
Markus Bachmayr, Albert Cohen, Dinh D\~ung and Christoph Schwab
}
\maketitle
\date{}
\begin{abstract}
It has recently been demonstrated that locality of spatial supports in the parametrization 
of coefficients in elliptic PDEs can lead to improved convergence rates of sparse polynomial expansions 
of the corresponding parameter-dependent solutions. 
These results by themselves do not yield practically realizable approximations, 
since they do not cover the approximation of the arising expansion coefficients, 
which are functions of the spatial variable. 
In this work, we study the combined spatial and parametric approximability for 
elliptic PDEs with affine or lognormal parametrizations of the diffusion coefficients and corresponding 
Taylor, Jacobi, and Hermite expansions, to obtain fully discrete approximations. 
Our analysis yields convergence
rates of the fully discrete approximation in terms of the total number of degrees of freedom.  
The main vehicle consists in $\ell^p$ summability results for the coefficient sequences measured 
in higher-order Hilbertian Sobolev norms. 
We also discuss similar results for non-Hilbertian Sobolev 
norms which arise naturally when using adaptive spatial discretizations.
\end{abstract}

\section{Introduction}
\label{sec:Intro}

Parametric PDEs are of interest for modeling many complex phenomena,
where the involved parameters may be of deterministic or stochastic nature.
Their numerical treatment was initiated in the 1990s, see \cite{GS1,GS2,KL,X} 
for general references. 
It has recently drawn much attention in the case where the number of involved parameters
is very large \cite{BNTT1,BNTT2}, or {\it countably infinite} \cite{CD,CDS,CCS}. 

In this paper, we consider the elliptic diffusion equation
\be
-{\rm div}(a\nabla u)=f,
\label{ellip}
\ee
set on a given bounded Lipschitz domain $D\subset \R^m$ (say with $m=1,2$ or $3$), 
for some fixed right-hand side $f$, homogeneous Dirichlet boundary conditions
$u_{|\partial D}=0$, and spatially variable scalar diffusion coefficient $a$. 
Using the notation $V=H^1_0(D)$ and $V'=H^{-1}(D)$,
for any $f\in V'$, the weak formulation of \iref{ellip} in $V$,
\be
\int_D a\nabla u\cdot \nabla v=\langle f,v\rangle_{V',V}, \quad v\in H^1_0(D),
\ee
has a unique solution $u\in V$ whenever the diffusion coefficient $a$ satisfies $0<r<a<R<\infty$.

We consider diffusion coefficients having a parametrized form $a=a(y)$, where $y=(y_j)_{j\geq 1}$
is a sequence of real-valued parameters ranging in some set $U\subset \R^\N$. 
The resulting solution map
\be
u\mapsto u(y),
\label{solmap}
\ee
acts from $U$ to the solution space $V$. The objective is to achieve numerical 
approximation of this complex map by a small number of parameters with some
guaranteed error in a given norm. Two relevant types of parametrizations have
been the object of intensive study.

The first one is the so-called affine form
\be
a=a(y)= \bar a+\sum_{j\geq 1} y_j\psi_j,
\label{affine}
\ee
where $\bar a$ and $(\psi_j)_{j\geq 1}$ are given functions in $L^\infty(D)$,
and the parameters $y_j$ range in $[-1,1]$. 
The parameter domain is thus
\be
U=[-1,1]^\N
\ee
and the solution map is well-defined under the so-called {\it uniform ellipticity assumption}
\be
\sum_{j\geq 1}|\psi_j| \leq \bar a-r, \quad \mbox{a.e. on } D,
\label{uea}
\ee
for some $r>0$. An equivalent condition is that 
$\bar a\in L^\infty(D)$ is such that $\essinf \bar a>0$ and 
\be
\left \|\frac {\sum_{j\geq 1}|\psi_j| }{\bar a}Ê\right \|_{L^\infty} =\theta <1.
\label{uea1}
\ee
In this case, the approximation error is often measured in $L^\infty(U,V)$ or $L^2(U,V,\sigma)$
where $d\sigma$ is some product probability measure on $U$.

The second one is the so-called lognormal form
\be
a=\exp(b), \quad b=b(y)=\sum_{j\geq 1} y_j\psi_j,
\label{lognormal}
\ee
where the $y_j$ are i.i.d. standard Gaussian random variables. 
In this case, we have $U=\R^\N$ and the error is measured in $L^2(U,V,\gamma)$ where
\be
d\gamma(y):=\bigotimes_{j\geq 1} g(y_j)\,dy_j, \quad g(t):=\frac 1 {\sqrt{2\pi}} e^{-t^2/2}.
\label{gaussmeasure}
\ee

Note that the above countably infinite-dimensional setting comprises
its finite-dimensional counterpart
by setting $\psi_j=0$ for $j$ large enough.
In both affine and lognormal cases, a powerful strategy for the approximation of the 
solution map is based on the truncation of polynomial expansions
of the general form
\be
u(y)=\sum_{\nu\in\cF} u_\nu \,\phi_\nu(y).
\label{series}
\ee
Here $\cF$ is the set of finitely supported sequences of non-negative integers,
and 
\be
\phi_\nu(y)=\prod_{j\geq 1}\vp_{\nu_j}(y_j), \quad \nu=(\nu_j)_{j\geq 1},
\ee
where $(\vp_n)_{n\geq 0}$ is a family of univariate polynomials such that $\vp_0=1$ and ${\rm deg}(\vp_n)=n$.
The approximation strategy consists in defining
\be
u_n(y)=\sum_{\nu\in\Lambda_n} u_\nu \,\phi_\nu(y),
\label{truncate}
\ee
where $\Lambda_n$ is a selection of $n$ indices from $\cF$. If the
$\phi_\nu$ are normalized so that $|\phi_\nu(y)|\leq 1$ for all $y\in U$, then
the truncation error can be controlled in $L^\infty(U,V)$ by
\be
\|u-u_n\|_{L^\infty(U,V)} \leq \sum_{\nu\notin\Lambda_n}  \|u_\nu\|_V.
\label{superror}
\ee
If the $\phi_\nu$ constitute an orthonormal basis of $L^2(U)$ for some
given measure, then the truncation error in $L^2(U,V)$ for the same measure
is given by 
\be
\|u-u_n\|_{L^2(U,V)} =\left (\sum_{\nu\notin\Lambda_n}  \|u_\nu\|_V^2\right )^{1/2}.
\label{ell2error}
\ee
These estimates justify the use of best $n$-term trunctations, that is, 
taking $\Lambda_n$ to be the index set corresponding to the $n$ largest $\|u_\nu\|_V$.
With such a choice, classical results on best $n$-term approximation
in sequence spaces \cite{De} show that if $0<p<q$,
\be
(\|u_\nu\|_V)_{\nu\in\cF} \in \ell^p(\cF) \implies \left (\sum_{\nu\notin\Lambda_n}  \|u_\nu\|_V^q\right )^{1/q}\leq C n^{-s}, \quad s:=\frac 1 p-\frac 1 q,
\ee
where $C:=\|(\|u_\nu\|_V)_{\nu\in\cF}\|_{\ell^p}$. In particular, the tails in the right-hand sides of \iref{superror} and \iref{ell2error}
decrease like $n^{-s}$ where $s=\frac 1 p-1$ and $s=\frac 1 p- \frac 1 2$
respectively, provided that $(\|u_\nu\|_V)_{\nu\in\cF} \in \ell^p(\cF)$, with $p<1$ and $p<2$, respectively. A central objective is therefore
to establish $\ell^p$ summability results for the coefficient sequence in the given expansion, with $p$ as small as possible.

In the affine case, we consider two types of expansions. The first type are the Taylor (or power) series of the form
\be
\sum_{\nu\in\cF} t_\nu y^\nu, \quad t_{\nu}:=\frac 1 {\nu !}\partial^\nu u(y=0), \quad \nu!:=\prod_{j\geq 1} \nu_j!,
\label{taylor}
\ee
with the convention that $0!=1$. The second type are the orthogonal Jacobi series of the form
\be
\sum_{\nu\in\cF} v_\nu J_\nu(y), \quad J_\nu(y)=\prod_{j\geq 1}J_{\nu_j}(y_j),\quad v_\nu:=\int_U u(y)J_\nu(y) d\sigma(y),
\label{legendre}
\ee
where $(J_k)_{k\geq 0}$ is the sequence of Jacobi polynomials on $[-1,1]$ 
normalized with respect to the Jacobi probability measure
\be
\int_{-1}^1 |J_k(t)|^2 d_{\alpha,\beta}(t)dt =1,
\ee
where $d_{\alpha,\beta}(t)=c_{\alpha,\beta}(1-t)^\alpha(1+t)^\beta$ and $c_{\alpha,\beta}=(\int_{-1}^1(1-t)^\alpha(1+t)^\beta dt)^{-1}$. Therefore
$(J_\nu)_{\nu\in \cF}$ is an orthonormal basis of $L^2(U,\sigma)$, where
\be
d\sigma(y):=\bigotimes_{j\geq 1} d_{\alpha,\beta}(y_j)\,dy_j
\label{multijacobimeasure}
\ee
One particular example corresponding to the values $\alpha=\beta=0$ and $d_{0,0}(t)=\frac 1 2$ is the family 
of tensorized Legendre polynomials.

In the lognormal case, we consider Hermite series of the form
\be
\sum_{\nu\in\cF} w_\nu H_\nu(y), \quad H_\nu(y)=\prod_{j\geq 1}H_{\nu_j}(y_j),\quad w_\nu:=\int_U u(y)\,H_\nu(y)\, d\gamma (y),
\label{hermite}
\ee
with $(H_k)_{k\geq 0}$ being the sequence of Hermite polynomials normalized according to
\be
\int_{\R} | H_k(t)|^2\, g(t)\,dt=1,
\ee 
and $d\gamma$ given by \iref{gaussmeasure}. In this case $U=\R^\N$ and
$(H_\nu)_{\nu \in \cF}$ is an orthonormal basis of $L^2(U,\gamma)$.

In the affine case, the first $\ell^p$ summability results have been obtained in \cite{CDS} for the 
Taylor and Legendre coefficient sequences, under the conditions that \iref{uea} 
holds and $(\|\psi_j\|_{L^\infty})_{j\geq 1} \inÊ\ell^p(\N)$, for some $0<p<1$.
In the lognormal case, similar results have been first obtained in
\cite{HS} for the Hermite coefficients under the conditions that 
$(j\|\psi_j\|_{L^\infty})_{j\geq 1} \inÊ\ell^p(\N)$, for some $0<p\leq 1$.
Such results yield algebraic convergence rates $n^{-s}$ that are free
from the curse of dimensionality, in the sense that they hold for countably
many variables $y_j$. One of their intrinsic limitations is that the conditions on the
functions $\psi_j$ are expressed only through their $L^\infty$ norms.
In particular, they do not take into account their support properties, and
the possible gain in summability when these supports have limited overlap.

A different approach to summability results that takes into account the 
support properties has recently been proposed in \cite{BCM} for the affine case
and \cite{BCDM} for the lognormal case. This approach gives significant improvements
on the provable range of $\ell^p$ summability when the $\psi_j$ have disjoint supports, or limited overlap,
such as splines, finite elements or wavelet bases. The main results are the following
for the affine and lognormal case, respectively.

\begin{theorem}
\label{theotaylorV}
Let $0<q<\infty$ and $0<p<2$ be such that $\frac 1 p=\frac 1 q +\frac 1 2$. 
Assume that $\bar a\in L^\infty(D)$ is such that $\essinf \bar a>0$,
and that there exists a sequence $\rho=(\rho_j)_{j\geq 1}$ of 
numbers strictly larger than $1$ 
such that $(\rho_j^{-1})_{j\geq 1}\in \ell^q(\N)$ and such that 
\be
\theta 
:=
\left \| \frac{\sum_{j\geq 1} \rho_j|\psi_j|}{\bar a} \right \|_{L^\infty}
<1\;.
\label{uearho0}
\ee 
Then $(\|t_\nu\|_V)_{\nu\in \cF}$ and $(\|v_\nu\|_V)_{\nu\in \cF}$ belong to $\ell^p(\cF)$.
\end{theorem}

\begin{theorem}
\label{theohermiteV}
Let $0<q<\infty$ and $0<p<2$ be such that $\frac 1 p=\frac 1 q +\frac 1 2$. 
Assume that there exists a sequence $\rho=(\rho_j)_{j\geq 1}$ of positive numbers 
such that $(\rho_j^{-1})_{j\geq 1}\in \ell^q(\N)$ and such that 
\be
\left\| \sum_{j\geq 1} \rho_j |\psi_j|\right\|_{L^\infty}  <\infty\;.
\ee
Then the sequence $(\|w_\nu\|_V)_{\nu\in \cF}$ belongs to $\ell^p(\cF)$.
\end{theorem}

While the above results allow us to establish algebraic convergence rates
of the properly truncated series \iref{truncate} in terms of the number $n$
of retained coefficients, they do not yet yield practically realizable approximations.
Indeed, the coefficients $u_\nu\in \{t_\nu,v_\nu,w_\nu\}$ belong to the infinite
dimensional space $V$ and should themselves be approximated by means
of spatial discretizations. Such discretizations are typically performed 
by means of finite elements or wavelets,
and the number $n_\nu$ of allocated degrees of freedom may vary with
the retained index $\nu$.

In the present paper, we address this issue and establish several
results that describe the rate of convergence in terms of the total number
of degrees of freedom $N=\sum_{\nu\in\Lambda_n} n_\nu$. Our main vehicle
is to investigate the $\ell^p$ summability of the sequence $(\|u_\nu\|)_{\nu\in \cF}$
for the various coefficients $u_\nu\in \{t_\nu,v_\nu,w_\nu\}$, where $\|\cdot\|$ is a
norm associated to a higher-order smoothness class in $V$, for example the Sobolev
space $H^k(D)$ for some $k>1$. These new summability results are stated in \S 2,
and their implications on the convergence rate of spatial-parametric discretization
methods are discussed in \S 3 using finite element spaces for the spatial discretization. 
The proof of the convergence results is given in \S 4-5-6 for the Taylor, Jacobi and Hermite
coefficients, respectively. The smoothness classes considered in these results are
Hilbertian Sobolev spaces. In order to analyze the potential benefit of 
a fully adaptive space-parameter discretization, it is interesting to consider
also non-Hilbertian Sobolev or Besov classes which are known to 
govern the rate of nonlinear approximation methods such as adaptive 
finite elements or wavelets. Some first results in this direction
are given in \S 7. 

Finally, we illustrate in \S 8 the various results in this paper for a specific 
affine parametrization of the diffusion coefficient by a wavelet decomposition.
For this specific example, one main finding is that the convergence
rate in terms of $N$ for the space-parameter approximation
is typically higher when using nonlinear approximation both in the spatial
and parametric variable.

The approximation results obtained in this paper may be viewed
as a benchmark for the performance of concrete numerical strategies,
such as non-adaptive or adaptive space-parameter Galerkin methods.

\section{Summability results}
\label{sec:MainRes}

The results established in the further sections \S 4-5-6-7 of this paper give sufficient conditions
for $\ell^p$ summability of the sequences $(\|u_\nu\|_X)_{\nu\in \cF}$ where
$u_\nu\in \{t_\nu,v_\nu,w_\nu\}$ and $\|\cdot\|_X$ is a norm of a relevant smoothness
class $X$ of $V=H^1_0(D)$.

One first such class that appears in a natural manner is the space
\be\label{Wdef}
W:=\{v\in V\; : \; \Delta v\in L^2(D)\}.
\ee
This space is equipped with the norm
\be
\label{Wnorm}
\|v\|_W:=\|\Delta v\|_{L^2}.
\ee
Note that the above is a norm since $\Delta v=0$ implies $v=0$ when $v\in V$ due 
to the boundary condition. By elliptic regularity theory, it is also known that
$W$ coincides with the Sobolev space $V\cap H^2(D)$ with equivalent norms,
in the case where the domain $D$ is convex or of $C^{1,1}$ smoothness,
see \cite[Theorems 3.2.1.3 and 2.4.2.5]{Gr}. 
Our first result, established in \S 4.1, concerns the Taylor coefficients in the 
case of the affine parametrization \iref{affine}. Here, and in the rest of the paper,
we use the notation
\be
|\nabla \psi|= \Bigl(\sum_{i=1}^m |D_{x_i} \psi|^2 \Bigr)^{1/2},
\ee
for the Euclidean norm of the gradient.

\begin{theorem}
\label{theotaylorH2}
Let $0<q<\infty$ and $0<p<2$ be such that $\frac 1 p=\frac 1 q +\frac 1 2$. 
Assume that $\bar a\in L^\infty(D)$  is such that $\essinf \bar a>0$,
and that there exists a sequence $\rho=(\rho_j)_{j\geq 1}$ of 
numbers strictly larger than $1$ 
such that $(\rho_j^{-1})_{j\geq 1}\in \ell^q(\N)$ and such that 
\be
\theta := 
\left \| \frac{\sum_{j\geq 1} \rho_j|\psi_j|}{\bar a} \right \|_{L^\infty} 
< 1 \;.
\label{uearho}
\ee 
Assume in addition that the right side $f$ in \iref{ellip}
belongs to $L^2(D)$ and that $\bar a$ and all functions $\psi_j$
belong to $W^{1,\infty}(D)$ and that
\be
\left\| \sum_{j\geq 1} \rho_j |\nabla \psi_j| \right\|_{L^\infty} <\infty\;.
\label{sumgradrho}
\ee
Then the sequence $(\|t_\nu\|_W)_{\nu\in \cF}$ belongs to $\ell^p(\cF)$. In particular,
when $D$ is convex or of $C^{1,1}$ smoothness, the same holds for $(\|t_\nu\|_{H^2})_{\nu\in \cF}$.
\end{theorem}

\begin{remark}
\label{remstringent} 
The first condition \iref{uearho} in the above theorem is the same as \iref{uearho0} in
Theorem \ref{theotaylorV}. However, the additional condition \iref{sumgradrho} puts further
constraints on the choice of the sequence $\rho$. For this reason, we expect that
the two values
$p_V:=\inf\{p>0\, : \: (\|t_\nu\|_V)_{\nu\in \cF}\in \ell^p(\cF)\}$
and $p_W:=\inf\{p>0\, : \: (\|t_\nu\|_W)_{\nu\in \cF}\in \ell^p(\cF)\}$
may in general differ in the sense that $p_V<p_W$. This remark also applies to
the next results dealing with higher-order smoothness and other types of polynomial expansions.
\end{remark}

Higher-order smoothness can be treated under more stringent conditions
on the smoothness of the functions $f$, $\bar a$, $\psi_j$ and of the domain $D$.
For any integer $k\geq 2$ we introduce the space
\be
W^k:=\{v\in V\; : \; \Delta v\in H^{k-2}(D)\}.
\ee
In particular $W^2=W$ defined above. This space is equipped with the norm
\be
\|v\|_{W^k}:=\|\Delta v\|_{H^{k-2}},
\ee
and coincides with the Sobolev space $V\cap H^k(D)$ with equivalent norms if the domain $D$ has $C^{k-1,1}$ smoothness,
see \cite[Theorem 2.5.1.1]{Gr}. The following result, established in \S 4.2, generalizes Theorem \ref{theotaylorH2}
to such spaces.

\begin{theorem}
\label{theotaylorHm}
Let $0<q<\infty$ and $0<p<2$ be such that $\frac 1 p=\frac 1 q +\frac 1 2$, and let $k> 2$ be an
integer.  Assume that $\bar a\in L^\infty(D)$  is such that $\essinf \bar a>0$,
and that there exists a sequence $\rho=(\rho_j)_{j\geq 1}$ of numbers strictly larger than $1$ 
such that $(\rho_j^{-1})_{j\geq 1}\in \ell^q(\N)$ and such that \iref{uearho}
holds. Assume in addition that the right side $f$ in \iref{ellip}
belongs to $H^{k-2}(D)$, that the domain $D$ has $C^{k-2,1}$ smoothness,
that $\bar a$ and all functions $\psi_j$ belong to $W^{k-1,\infty}(D)$ and that
\be
\sup_{|\alpha|\leq k-1} 
\left\| \sum_{j\geq 1} \rho_j |D^\alpha \psi_j| \right\|_{L^\infty} 
<\infty\;.
\label{sumderrho}
\ee
Then the sequence $(\|t_\nu\|_{W^k})_{\nu\in \cF}$ belongs to $\ell^p(\cF)$. 
In particular, when $D$ has $C^{k-1,1}$ smoothness, 
the same holds for $(\|t_\nu\|_{H^k})_{\nu\in \cF}$.
\end{theorem}

We also present in \S 4.3 some variants of these results
using fractional Sobolev spaces and
weighted Sobolev spaces for polygonal domains. 
In the case of Jacobi expansions, the following analogous results are
established in \S 5 under the exact same assumptions.

\begin{theorem}
\label{theojacobiHm}
The sequence $(\|v_\nu\|_{W^k})_{\nu\in \cF}$ belongs to $\ell^p(\cF)$,
under the same assumptions as those of Theorem \ref{theotaylorH2}
when $k=2$ and of Theorem \ref{theotaylorHm} when $k>2$.
\end{theorem}

Finally, in the case of the lognormal parametrization
\iref{lognormal}, we establish in \S 6 the following analog to Theorem \ref{theotaylorH2}.

\begin{theorem}
\label{theohermiteH2}
Let $0<q<\infty$ and $0<p<2$ be such that $\frac 1 p=\frac 1 q +\frac 1 2$. 
Assume that the right side $f$ in \iref{ellip}
belongs to $L^2(D)$ and that $\bar a$ and all functions $\psi_j$
belong to $W^{1,\infty}(D)$. Assume in
addition that there exists a sequence $\rho=(\rho_j)_{j\geq 1}$ of positive numbers 
such that $(\rho_j^{-1})_{j\geq 1}\in \ell^q(\N)$ and such that 
\be
\left\| \sum_{j\geq 1} \rho_j |\psi_j|\right\|_{L^\infty} 
+
\left\| \sum_{j\geq 1} \rho_j |\nabla \psi_j|\right\|_{L^\infty} <\infty.
\label{condhermite}
\ee
Then $u\in L^k(U,W,\gamma)$ for all $k>0$ and
the sequence $(\|w_\nu\|_W)_{\nu\in \cF}$ belongs to $\ell^p(\cF)$. In particular,
when $D$ is convex or of $C^{1,1}$ smoothness, the same holds for $(\|w_\nu\|_{H^2})_{\nu\in \cF}$.
\end{theorem}

The proof of this result is given in \S 6.
For higher-order smoothness, an analogous result seems to be valid, namely:
under the assumption that the domain $D$ has $C^{k-2,1}$ smoothness,
that $\bar a$ and all functions $\psi_j$
belong to $W^{k-1,\infty}(D)$ and that \iref{sumderrho} holds,
then the sequence $(\|w_\nu\|_{W^k})_{\nu\in \cF}$ belongs to $\ell^p(\cF)$. 
However, the proof of this result seems to require heavy technical 
and notational developments, and we therefore do not attempt to include it in this paper.

\section{Space-parameter approximation}
\label{sec:SpcParDiscr}

We recall that polynomial approximations are typically obtained
by truncation of the general series \iref{series} according to
\be
u_n:=\sum_{\nu\in\Lambda_n} u_\nu \, \phi_\nu.
\label{nterm}
\ee
When taking for $\Lambda_n$ a set corresponding to the $n$ largest $\|u_\nu\|_V$,
the rate of convergence of such approximations
is governed by the $\ell^p$ summability of the $(\|u_\nu\|_V)_{\nu\in \cF}$.
However, the approximation $u_n$ is still picked from an infinite-dimensional
space, namely
\be
V_{\Lambda_n}:=V \otimes {\rm span}\{\phi_\nu\; : \; \nu\in \Lambda_n\}.
\ee
The summability results for the $W^k$ or $H^k$ norms of the coefficients $u_\nu$ stated 
in the previous section allow us to introduce further spatial discretization resulting 
in approximations picked from finite-dimensional spaces, and 
to analyze the rate of convergence in terms of the dimension of such spaces.
As already noted in Remark \ref{remstringent}, since the conditions in these results become more stringent as $k$ increases, we generally obtain
that
\be
(\|u_\nu\|_{H^k})_{\nu\in \cF}\in \ell^{p_k}(\cF),
\ee
for some sequence
\be
p_1\leq p_2 \leq \dots \leq p_k.
\ee
Generally speaking, we assume that 
\be
(\|u_\nu\|_{V})_{\nu\in \cF}\in \ell^{p_V}(\cF),
\ee
and that for some given regularity class $X\subset V$ of interest,
we have that 
\be
(\|u_\nu\|_{X})_{\nu\in \cF}\in \ell^{p_X}(\cF),
\label{px}
\ee
for some $p_V\leq p_X$.

The spatial discretization is based on a sequence 
$(V_n)_{n>0}$ of subspaces of $V$ with dimension
\be
\dim(V_n)=n.
\ee
One instance is obtained by considering Lagrange finite elements 
on a family of quasi-uniform regular simplicial partitions $(\cT_h)_{h>0}$ 
of mesh size $h$, up to isoparametric transformations in the case of curved domains.
Since $D$ is a bounded domain of $\R^m$, 
the dimension $n$ of $V_n$ is related to the corresponding mesh size
\be
n\sim h^{-m},
\ee
We assume that such spaces have the following \emph{approximation property}: there is
a constant $C_X>0$ and $t>0$, 
\be
\min_{v_n\in V_n} \|v-v_n\|_V \leq C_Xn^{-t} \|v\|_{X}, \quad n > 0, \quad v\in X,
\label{spatialappn}
\ee
where $X\subset V$ is the regularity class for which \iref{px}Ê holds.

For example, with $X=H^k(D)$, classical error estimates \cite{Cia} yield
the convergence rate $t=\frac{k-1}m$ by using Lagrange finite elements of order at least $k-1$
on quasi-uniform partitions. Note that the spaces $W^k$ introduced in the
previous section do not always coincide with $H^k(D)$. 
For example, for $k=2$ and dimension $m=2$, we know that $W$ is strictly larger than $H^2(D)$
when $D$ is a polygon with re-entrant corner. In this case, it is well known that 
the optimal rate $t=1/2$ is yet attained with $X= W$, 
when using spaces $V_n$ associated to meshes $(\cT_n)_{n>0}$ with proper refinement near 
the re-entrant corners where the functions $v\in W$ might have singularities.

The relevant finite element discretizations are therefore typically generated by uniform 
or non-uniform refinement from some regular, initial triangulation of $D$. 
Note that not all values of $n>0$ may be realized in this process. 
However, up to completing the sequence by $V_n:=V_{\tilde n}$,
where $\tilde n$ is the largest attained value below $n$, we reach the same approximation
estimate \iref{spatialappn} with spaces $V_n$ of dimension at most $n$.

We apply spatial discretization in \iref{nterm} by replacing each $u_\nu\in V$ by some
$u_{\nu,n_\nu}\in V_{n_\nu}$, where we allow the number of degrees of freedom $n_\nu$ 
to depend on $\nu$.
This variability is critical for the resulting convergence rate
in terms of the total number of degrees of freedom.
Thus, with
\be\label{bndof}
\bn:=(n_\nu)_{\nu\in\Lambda_n},
\ee 
the resulting approximant
\be
u_{\bn}:=\sum_{\nu\in\Lambda_n} u_{\nu,n_\nu} \phi_\nu,
\label{ntermh}
\ee
belongs to the space
\be
V_{\bn}:=\bigoplus_{\nu\in\Lambda_n} (V_{n_\nu}Ê\otimes \RÊ\phi_\nu).
\label{vbn}
\ee
This space is characterized by the total number of degrees of freedom for each $V_{n_\nu}$:
\be
N=\dim(V_{\bn})=\sum_{\nu\in \Lambda_n} n_{\nu}.
\ee
We again take for $\Lambda_n$ a set corresponding to 
$n$ largest $\|u_\nu\|_{V}$.

If the coefficients $u_\nu$ belong to the smoothness space $X$, 
there exist $u_{\nu,n_\nu}\in V_{n_\nu}$ such that
\be
\|u_\nu-u_{\nu,n_\nu}\|_V \leq C_X n_\nu^{-t} \|u_\nu\|_X.
\ee
We may in particular take for $u_{\nu,n_\nu}$ the $V$-orthogonal projection
of $u_\nu$ onto $V_{n_\nu}$.

We may then write for a given norm $\cV$, 
\be
\|u-u_{\bn}\|_{\cV} \leq \|u-u_n\|_{\cV} +\Big\|\sum_{\nu\in \Lambda_n} (u_\nu-u_{\nu,n_\nu})\phi_\nu\Big\|_{\cV}.
\ee
In the case of Taylor series (or more general polynomial series where the $\phi_\nu$ are 
uniformly bounded by $1$ over $U$), taking $\cV=L^\infty(U,V)$, we control the first term by
\be\label{ntermconvergence}
\|u-u_n\|_{\cV} \leq \sum_{\nu\notin \Lambda_n} \|u_\nu\|_V \leq Cn^{-s}, \quad s:=\frac 1 {p_V}-1, \quad C:=\|(\|u_\nu\|_V)_{\nu\in\cF} \|_{\ell^{p_V}},
\ee
and the second term by
\be
\Big\|\sum_{\nu\in \Lambda_n} (u_\nu-u_{\nu,n_\nu})\phi_\nu \Big\|_{\cV}\leq C_X\sum_{\nu\in \Lambda_n} n_\nu^{-t} \|u_\nu\|_{X}.
\ee
We now allocate the degrees of freedom, that is, the values of $n_\nu$, in order to minimize $N$
for a given total error. Leaving aside the multiplicative constants, this amounts to solving the constrained minimization problem
\be
\min \Big\{\sum_{\nu\in \Lambda_n} n_\nu \; : \; \sum_{\nu\in \Lambda_n} n_\nu^{-t} \|u_\nu\|_{X}\leq n^{-s} \Big\}.
\ee
Up to introducing multiplicative constants, we treat the $n_\nu$ as real numbers, and using a Lagrange multiplier we obtain
\be
n_\nu=\eta \|u_\nu\|_{X}^{\frac 1 {1+t}},
\label{optimnu}
\ee
for some $\eta>0$ independent of $\nu\in \Lambda_n$. 
Its value is determined by the saturated constraint
\be
n^{-s}= \sum_{\nu\in \Lambda_n}n_\nu^{-t} \|u_\nu\|_{X}\
=
\eta^{-t}\sum_{\nu\in \Lambda_n} \|u_\nu\|_{X}^{\frac 1 {1+t}},
\label{saturated}
\ee
and therefore
\be
\eta=n^{\frac s t}\Big(\sum_{\nu\in \Lambda_n} \|u_\nu\|_{X}^{\frac 1 {1+t}}\Big)^{\frac 1 t}.
\label{eta}
\ee
Combining this with \iref{optimnu} and summing over $\nu\in\Lambda_n$, we find
\be
N=n^{\frac s t}\Big(\sum_{\nu\in \Lambda_n} \|u_\nu\|_{X}^{\frac 1 {1+t}}\Big)^{\frac {1+t} t}.
\label{dofN}
\ee
We now distinguish between two cases.
\begin{enumerate}
\item
$p_X\leq \frac 1 {t+1}$ in which case
$\sum_{\nu\in \Lambda_n}  \|u_\nu\|_{X}^{\frac 1 {1+t}}$
is bounded independently of the set $\Lambda_n$. 
Since the global error is controlled by $n^{-s}$, 
we obtain a bound of the form
\be
\|u-u_{\bn}\|_{\cV} \leq CN^{-t}.
\label{globalxyerror3}
\ee
Note that this is the convergence rate for the spatial discretization 
of a single $v\in X$ as given by \iref{spatialappn}.
\item
$p_X> \frac 1 {t+1}$: 
in this case $\sum_{\nu\in \Lambda_n}\|u_\nu\|_{X}^{\frac 1 {1+t}}$
may not be uniformly bounded and we estimate it using H\"older's inequality
that gives
$$
\sum_{\nu\in \Lambda_n} \|u_\nu\|_{X}^{\frac 1 {1+t}}
\leq Cn^\delta, \quad \delta:=1-\frac 1 {p_X(1+t)} > 0\;.
$$
According to \iref{dofN}, we thus have
\be
N \leq  Cn^{\frac {s+(1+t)\delta} t}\;.
\ee
Combining this with the fact that the global error is controlled by $n^{-s}$,  
we obtain a bound of the form
\be
\|u-u_{\bn}\|_{\cV} \leq CN^{-r}, \quad r:=\frac {st}{s+(1+t)\delta}.
\label{globalxyerror4}
\ee
Note that 
\be
r=\frac {st}{t+s+1-p_X^{-1}} \leq \frac {st}{t+s+1-p_V^{-1}}=s.
\ee
On the other hand, since the second expression increases with $s$ 
as long as $s\geq p_X^{-1}-1$, 
and since the actual value of $s$ is $p_V^{-1}-1\geq p_X^{-1}-1$, 
we find that
$r\geq \frac 1 {p_X}-1$. In summary, we find that
\be
\frac 1 {p_X}-1\leq r \leq \frac 1 {p_V}-1.
\ee
\end{enumerate}

In the case of Jacobi or Hermite series 
(or more general orthonormal polynomial expansions),
taking $\cV=L^2(U,V)$ with the appropriate probability measure, 
we control the first term by
\be
\|u-u_n\|_{\cV} = \left(\sum_{\nu\notin \Lambda_n} \|u_\nu\|_V^2\right)^{1/2} \leq Cn^{-s}, \quad s:=\frac 1 {p_V}-\frac 1 2, \quad C:=\|(\|u_\nu\|_V)_{\nu\in\cF} \|_{\ell^{p_V}},
\ee
and the second term by
\be
\Big\|\sum_{\nu\in \Lambda_n} (u_\nu-u_{\nu,n_\nu})\phi_\nu \Big\|_{\cV}
=\left(\sum_{\nu\in \Lambda_n}\|u_\nu-u_{\nu,n_\nu}\|_V^2\right)^{1/2}
\leq C_X\left(\sum_{\nu\in \Lambda_n} n_\nu^{-2t} \|u_\nu\|_{X}^2\right)^{1/2}.
\ee
Proceeding in a similar way as in the previous case for the optimal allocation of 
the degrees of freedom, we now obtain
\be
N=n^{\frac s t}\Big(\sum_{\nu\in \Lambda_n} \|u_\nu\|_{X}^{\frac 2 {1+2t}}\Big)^{\frac {1+2t} {2t}}.
\label{dofN2}
\ee
This leads to the following two cases which are discussed in the same way as before:
\begin{enumerate}
\item
$p_X\leq \frac 2 {2t+1}$ in which case we obtain the global error bound
\iref{globalxyerror3}.
\item
$p_X> \frac 2 {2t+1}$ in which case we obtain the global error bound
\iref{globalxyerror4}, with rate now given by
\be
r:=\frac {st}{s+t+\frac 1 2-p_X^{-1}},
\label{rsecondcase}
\ee
which satisfies
\be
\frac 1 {p_X}-\frac 1 2\leq r\leq \frac 1 {p_V}-\frac 1 2.
\ee
\end{enumerate}

We summarize the above discussion with the following result.

\begin{theorem}\label{theoconvrate}
Assume that the discretization spaces $(V_n)_{n>0}$ satisfy the approximation property \iref{spatialappn}
for a certain smoothness class $X\subset V$
and that $(\|u_\nu\|_{V})_{\nu\in \cF}\in\ell^{p_V}(\cF)$ and $(\|u_\nu\|_{X})_{\nu\in \cF}\in\ell^{p_X}(\cF)$ for some $0<p_V\leq p_X <2$. 
Then, for each $n$
there exists a vector $\bn:=(n_\nu)_{\nu\in\Lambda_n}$ such that
\be
\min_{u_{\bn}\in V_{\bn}} \|u-u_{\bn}\|_{\cV} \leq CN^{-\min\{r,t\}},
\label{spaceparrate}
\ee
where $N=\sum_{\nu\in\Lambda_n} n_\nu=\dim(V_\bn)$.
Here $\cV$ is $L^\infty(U,V)$ in the case of Taylor series, $L^2(U,V,\sigma)$ in the case
of Jacobi series and $L^2(U,V,\gamma)$ in the case of Hermite series. 
The rate $t$ corresponds to the spatial approximation of a single $v\in X$ as given by \iref{spatialappn}.
The rate $r$ is given by \iref{globalxyerror4} in the first case
and by \iref{rsecondcase} in the second case.
The constant $C$ in \iref{spaceparrate} depends on
the quantities $\|(\|u_\nu\|_V)_{\nu\in\cF} \|_{\ell^{p_V}}$
and $\|(\|u_\nu\|_X)_{\nu\in\cF} \|_{\ell^{p_X}}$, as well as on
the constant $C_X$ in \iref{spatialappn}.
\end{theorem}

\begin{remark}
The idea of optimizing the allocation of spatial degrees of freedom for the different
coefficients $u_\nu$ in the expansion \iref{series} has also been 
developed in \cite{Di15} using a slightly different approach.
Here, one uses a nested hierarchy $(V_{2^k})_{k\geq 0}$ of space discretization
and $P_{V_{2^k}}$ denotes a $V$-stable projector on the space $V_{2^k}$, so that
the approximation property \iref{spatialappn} now takes the form
\be
\|v-P_{V_{2^k}} v\|_V \leq C_X 2^{-kt} \|v\|_{X}, \quad k \geq 0\;.
\label{spatialappk}
\ee
The approximations of the solution map are of the form
\be
u_G=\sum_{(k,\nu)\in G} \delta_k(u_\nu) \phi_\nu(y)
\ee
where $G$ is a finite subset of $\N_0\times \cF$, and where the $\delta_k(u_\nu)$
are multiscale components defined by
\be
\delta_k(v):=P_{V_{2^k}}v-P_{V_{2^{k-1}}}v, \quad k\geq 1, \quad \delta_0(v):=P_{V_1}v.
\ee
In \cite{Di15}, the proposed set $G$ 
consists of pairs $(k,\nu)$ that satisfy
\be
2^k\leq \alpha_\nu,
\label{knu}
\ee
for some appropriately chosen $(\alpha_\nu)_{\nu\in\cF}$. This leads to an approximation of the
form
\be
u_G=\sum_{\nu\in\Lambda} P_{V_{2^{k(\nu)}}} u_\nu\, \phi_\nu(y),
\ee
where $\Lambda$ is the set of $\nu$ such that $\alpha_\nu\geq 1$ and $k(\nu)$ is the largest $k$
such that \iref{knu} holds. This approximation is therefore similar to that proposed in \iref{ntermh},
and optimizing the choice of the sequence $(\alpha_\nu)_{\nu\in\cF}$ leads to similar 
convergence rates as given in Theorem \ref{theoconvrate}.
\end{remark}

\section{Summability of Taylor expansions}
\label{sec:TayExp}
\subsection{$H^2$-regularity}
\label{sec:H2Reg}
Establishing the $\ell^p$ summability 
in Theorem \ref{theotaylorH2} will be based 
on a weighted $\ell^2$ estimate expressed in the following result.
\begin{theorem}
\label{theotaylorwH2}
Assume that $\bar a\in L^\infty(D)$  is such that $\essinf \bar a>0$,
and that there exists a sequence $\rho=(\rho_j)_{j\geq 1}$ of positive numbers
such that \iref{uearho} holds. 
Assume in addition that the right side $f$ in \iref{ellip}
belongs to $L^2(D)$ and that $\bar a$ and all functions $\psi_j$
belong to $W^{1,\infty}(D)$ and that \iref{sumgradrho} holds. 
Then
\be
\sum_{\nu\in\cF} (\rho^{\nu} \|t_\nu\|_W)^2 <\infty, \quad 
\rho^\nu:=\prod_{j\geq 1} \rho_j^{\nu_j}\;.
\label{taylorwH2}
\ee
\end{theorem}
Before giving the proof of this theorem,
let us explain why it implies Theorem \ref{theotaylorH2}.
With $0<q<\infty$ and $0<p<2$ such that $\frac 1 p=\frac 1 q +\frac 1 2$, we obtain by H\"older's inequality
that 
\be
\left (\sum_{\nu\in\cF} \|t_\nu\|_W^p\right )^{1/p} \leq  \left (\sum_{\nu\in\cF} (\rho^{\nu} \|t_\nu\|_W)^2\right )^{1/2}
\left (\sum_{\nu\in\cF} \rho^{-q\nu} \right )^{1/q}.
\label{holder}
\ee
From the factorization
\be
\sum_{\nu\in\cF} \rho^{-q\nu}=\prod_{j\geq 1} \sum_{n\geq 0} \rho_j^{-nq}=\prod_{j\geq 1}\frac {1}{1-\rho_j^{-q}},
\ee
we easily conclude that the second factor on the right of \iref{holder} is finite under the assumption that
the $\rho_j$ are strictly larger than $1$ and $(\rho_j^{-1})_{j\geq 1}\in \ell^q(\N)$. 
Therefore $(\|t_\nu\|_W)_{\nu\in\cF}$ belongs to $\ell^p(\cF)$.

In the next proof we use the observation that for a given instance of
the elliptic PDE \iref{ellip}, when $f\in L^2$ and $a\in W^{1,\infty}$ is a strictly positive function, 
the equation can be written in the strong form
\be
-a\Delta u-\nabla a\cdot \nabla u=f,
\label{strongellip}
\ee
where all terms in this identity belong to $L^2(D)$.
\nl
\nl
{\bf Proof of Theorem \ref{theotaylorwH2}:} 
We first observe that it suffices to prove the result in the
case where $\rho_j=1$ for all $j\geq 1$. 
Indeed, the condition \iref{uearho} for a general positive
sequence $\rho$ is equivalent to the same condition with all $\rho_j=1$ when the functions $\psi_j$ 
are replaced by the functions $\rho_j\psi_j$. This amounts to considering the solution map
\be
y\mapsto \tilde u(y):=u(\rho y), \quad \rho y:=(\rho_j y_j)_{j\geq 1}.
\ee
Therefore, the result established in this particular case gives that $\sum_{\nu\in\cF} \|\tilde t_\nu\|_{W}^2<\infty$,
for the Taylor coefficients of $\tilde u$ which are given by $\tilde t_\nu=\rho^\nu t_\nu$, 
which is equivalent to \iref{taylorwH2}. 
We thus next assume that $\rho_j=1$ for all $j$ and establish $\ell^2$ summability of the 
$\|t_\nu\|_W$, under the condition that for some $0<\theta < 1$ holds
\be
\left\| \sum_{j\geq 1} \frac{|\psi_j|}{\bar{a}} \right\|_{L^\infty} \leq \theta \;.
\label{thetain}
\ee
We first observe that since $0<\bar a_{\min} \leq \bar a \leq \bar a_{\max}<\infty$, 
we have the norm equivalences
\be
\bar a_{\min}  \|v\|_V^2 \leq \int_D \bar a |\nabla v|^2\leq \bar a_{\max}  \|v\|_V^2,
\ee
and
\be
\bar a_{\min}  \|v\|_W^2 \leq \int_D \bar a |\Delta v|^2\leq \bar a_{\max}  \|v\|_W^2,
\ee
It will be more convenient to work with the above equivalent norms for $V$ and $W$. We 
thus introduce the quantities
\be
d_\nu:=\int_D \bar a |\nabla t_\nu|^2 \quad {\rm and} \quad c_\nu:=\int_D \bar a |\Delta t_\nu|^2,
\ee
and prove that $\sum_{\nu\in \cF}c_\nu<\infty$. For this purpose, we also introduce the quantities
\be
D_n:=\sum_{|\nu|=n} d_\nu\quad {\rm and} \quad C_n:=\sum_{|\nu|=n} c_\nu.
\ee
We first recall the argument from \cite{BCM} showing that $\sum_{\nu\in \cF}d_\nu<\infty$.
Applying $\frac 1 {\nu !}\partial^\nu$ at $y=0$ to the variational formulation
\be
\int_D a(y) \nabla u(y)\cdot \nabla v= \langle f,v\rangle_{V',V}, \quad v\in V, 
\ee
yields the recursion identity
\be
\int_D \bar a \nabla t_\nu\cdot \nabla v=-\sum_{j\in {\rm supp}(\nu)} \int_D \psi_j \nabla t_{\nu-e_j}\cdot\nabla v, \quad v\in V, 
\ee
with $e_j=(0,\dots,0,1,0,\dots)$ the Kronecker sequence of index $j$. Taking $v=t_\nu$ as a test function
and applying Young's inequality gives
\be
d_\nu \leq \frac 1 2\sum_{j\in {\rm supp}(\nu)} \int_D |\psi_j| |\nabla t_{\nu}|^2+\frac 1 2\sum_{j\in {\rm supp}(\nu)} \int_D |\psi_j| |\nabla t_{\nu-e_j}|^2.
\ee
From \iref{thetain}, the first term on the right does not exceed $\frac \theta 2 d_\nu$. 
Summing over $|\nu|=n$, we thus obtain
\begin{align*}
D_n & \leq \frac \theta 2 D_n+\frac 1 2\sum_{|\nu|=n}\sum_{j\in {\rm supp}(\nu)} \int_D |\psi_j| |\nabla t_{\nu-e_j}|^2\\
&= \frac \theta 2 D_n+\frac 1 2\sum_{|\nu|=n-1}\int_D \Bigl(\sum_{j\geq 1}|\psi_j| \Bigr)|\nabla t_{\nu}|^2\\
& \leq  \frac \theta 2 D_n+ \frac \theta 2 D_{n-1},
\end{align*}
where we have again used \iref{thetain}. This shows that
\be 
\label{eq:DnBd}
D_n\leq \kappa D_{n-1}\leq\dots \leq \kappa^n D_0, \quad \kappa=\frac{\theta}{2-\theta}<1,
\ee
with $D_0=d_0=\int_{D} \bar a |\nabla u(0)|^2\leq \frac {\|f\|_{V'}^2}{\bar a_{\min}}<\infty$,
and in particular that $\sum_{\nu\in\cF} d_\nu=\sum_{n\geq 0} D_n<\infty$.

We next turn to estimating the quantities $C_n$. For this we observe that for any $y\in U$
the function $a(y)$ belongs to $W^{1,\infty}(D)$. This allows to use the strong form \iref{strongellip}
\be
-a(y)\Delta u(y)-\nabla a(y)\cdot \nabla u(y)=f,
\ee
where equality holds in $L^2(D)$. 
Applying $\frac 1 {\nu !}\partial^\nu$ at $y=0$, for any $\nu\neq 0$, gives
\be 
\label{DeltRec}
-\bar a \Delta t_\nu=\nabla \bar a \cdot \nabla t_\nu+\sum_{j\in {\rm supp}(\nu)} (\psi_j \Delta t_{\nu-e_j}+\nabla \psi_j\cdot \nabla t_{\nu-e_j}),
\ee
which shows by recursion that $\Delta t_\nu\in L^2(D)$ for all $\nu\in \cF$.
Integrating against $\Delta t_\nu$, and applying Young's inequality we find that, for any $\e >0$,
\begin{align*}
c_\nu & \leq \frac 1 2\sum_{j\in {\rm supp}(\nu)} \int_D |\psi_j| |\Delta t_{\nu}|^2
+\e \sum_{j\in {\rm supp}(\nu)} \int_D |\nabla \psi_j| |\Delta t_{\nu}|^2+\e \int_D |\nabla \bar a| |\Delta t_{\nu}|^2\\
& 
+\frac 1 2\sum_{j\in {\rm supp}(\nu)} \int_D |\psi_j| |\Delta t_{\nu-e_j}|^2+\frac 1 {4\e}\sum_{j\in {\rm supp}(\nu)} \int_D |\nabla \psi_j| |\nabla t_{\nu-e_j}|^2 +\frac 1 {4\e} \int_D |\nabla \bar a| |\nabla t_{\nu}|^2 .
\end{align*}
The first three terms on the right are bounded by $(\frac \theta 2+\e B)c_\nu$, where
\be\label{defB}
B:=\Big \| |\nabla \bar a|+\sum_{j} |\nabla \psi_j| \Big \|_{L^\infty}<\infty.
\ee
Summing over $|\nu|=n$, and exchanging the summations in $j$ and $\nu$ in the last three terms,
we thus obtain
\be
C_n\leq (\frac \theta 2+\e B)C_n + \frac \theta 2 C_{n-1}+\frac {B}{4\e \bar a_{\min}} (D_n+D_{n-1}).
\ee
We next choose $\e>0$ small enough so that $\frac \theta 2+\e B<\frac 1 2$, so that
\be
C_n\leq \tau C_{n-1}+ A(D_n+D_{n-1}), \quad \tau:=\frac {\theta}{2-\theta -2\e B}<1, \quad A:=\frac {B}{2\e \bar a_{\min}(2-\theta -2\e B)}.
\ee
Since we have already seen that $D_n\leq D_0 \kappa^n$, we find that
\be
C_n\leq \tau C_{n-1} +AD_0(1+\kappa^{-1}) \kappa^n.
\ee
We now choose a $\delta$ such that
\be
\kappa\leq \tau<\delta <1,
\ee
and observe that $C_{n-1}\leq C\delta^{n-1}$ implies that
$C_n \leq (\tau \delta^{-1} C+AD_0(1+\kappa^{-1})) \delta^n$. We thus find by induction that
\be
C_n\leq \max\Big\{C_0,\frac {AD_0(1+\kappa^{-1})}{1-\tau \delta^{-1}}\Big\} \delta^n,
\ee
and in particular that $\sum_{\nu\in\cF} c_\nu=\sum_{n\geq 0} C_n<\infty$. \hfill $\Box$

\subsection{Higher-order regularity}

Similar to $W$ or $H^2$ regularity, the proof of $\ell^p$ summability 
in Theorem \ref{theotaylorHm} can be derived by the same argument from weighted $\ell^2$ estimate expressed
in the following result, analogous to Theorem \ref{theotaylorwH2}.

\begin{theorem}
\label{theotaylorwHm}
Assume that $\bar a\in L^\infty(D)$  is such that $\essinf \bar a>0$, and
that there exists a sequence $\rho=(\rho_j)_{j\geq 1}$ of positive numbers
such that \iref{uearho} holds.  Assume in addition that the right side $f$ in \iref{ellip}
belongs to $H^{k-2}(D)$ for integer $k\geq 2$, 
that the domain $D$ has $C^{k-2,1}$ smoothness,
that $\bar a$ and all functions $\psi_j$
belong to $W^{k-1,\infty}(D)$, and that \iref{sumderrho} holds. 
Then
\be
\sum_{\nu\in\cF} (\rho^{\nu} \|t_\nu\|_{W^k})^2 <\infty, \quad \rho^\nu:=\prod_{j\geq 1} \rho_j^{\nu_j}.
\ee
\end{theorem}

In proving this result, we use the following fact on elliptic regularity for a single instance
of the elliptic PDE \iref{ellip}. Here, we use $D^\alpha$ to denote the partial derivative
of order $\alpha$ in the space variable $x$, in order to avoid confusion with the partial
derivative $\partial^\nu$ with respect to the parametric variable $y$.

\begin{lemma}
Let $k\geq 2$ be an integer and assume that $a\in W^{k-1,\infty}(D)$ is a strictly positive function, 
$f\in H^{k-2}(D)$ and that the domain $D$ has $C^{k-2,1}$ smoothness. Then it holds that $u\in H^{k-1}(D)$ and
that for all $|\alpha|\leq k-2$, the equation
\be
\label{strongalpha}
-a\Delta D^\alpha u-\sum_{\beta\leq \alpha} {\alpha\choose \beta} \nabla D^\beta a \cdot \nabla D^{\alpha-\beta}Êu
-\sum_{\beta\leq \alpha, \beta \neq 0} {\alpha\choose \beta} D^\beta a \Delta D^{\alpha-\beta}Êu
= D^\alpha f,
\ee
holds in $L^2(D)$, where ${\alpha\choose \beta}$ denotes the usual binomial coefficient of multi-indices.
\end{lemma}

\noindent
{\bf Proof:} We proceed by induction on $k$. The case $k=2$ is just the strong form \iref{strongellip}.
For $k>2$, assuming the result at order $k-1$, we find that $D^\alpha\Delta u\in L^2$ for all $|\alpha|\leq k-3$,
or equivalently $\Delta u\in H^{k-3}(D)$. Since $D$ has $C^{k-2,1}$ smoothness,
and $f\in  H^{k-2}(D)\subset H^{k-3}(D)$, classical elliptic regularity theory \cite[Theorem 2.5.1.1]{Gr} 
implies that $u\in H^{k-1}(D)$. Then, since $a\in W^{k-1,\infty}(D)$ we may use the Leibniz rule
when applying $D^\alpha$ to \iref{strongellip} for $|\alpha|\leq k-2$ and we obtain \iref{strongalpha}
where each term belongs to $L^2(D)$.
\hfill $\Box$
\nl
\nl
{\bf Proof of Theorem \ref{theotaylorwHm}:} We proceed by induction on $k$,
reducing again to the case where $\rho_j=1$ for all $j$, as in the proof of Theorem \ref{theotaylorwH2} which 
gives the case $k=2$. Since the arguments in the induction step are in essence similar
to those used in the proof of this previous theorem, up to more involved notations, we mainly 
sketch them. 

For a given order $\alpha$ of spatial differentiation, we introduce
\be
c_{\nu}^\alpha:=\int_D \bar a |\Delta D^\alpha t_\nu|^2,
\ee
and prove that $\sum_{\nu\in \cF}c_\nu^\alpha<\infty$ for all $|\alpha|=k-2$. For this purpose, we introduce the quantities
\be
C_n^\alpha:=\sum_{|\nu|=n} c_\nu^\alpha,
\ee
and prove by induction on $k$ that 
\be
C_n^\alpha\leq M_{|\alpha|} \kappa_{|\alpha|}^n, \quad n\geq 0,
\ee
for some constant $M_k$ and $0<\kappa_k<1$ that can be assumed to be non-decreasing with $k$. For $k=2$, 
this was established in the proof of Theorem \ref{theotaylorwH2}. 
Note that the induction assumption for all $|\alpha|\leq k-3$, combined with the 
fact that $D$ has $C^{k-2,1}$ smoothness, also means, by elliptic regularity, that
\be
\sum_{|\nu|=n} \|t_\nu\|_{H^{k-1}(D)}^2 
=\sum_{|\nu|=n} \sum_{|\beta| \leq k-1}\|D^\beta t_\nu\|_{L^2(D)}^2\leq M_{k-1} \kappa_{k-1}^n,
\label{induct}
\ee
up to a multiplicative change in the constants $M_{k-1}$.

Using the strong form \iref{strongalpha} for $a=a(y)$ and $u=u(y)$, and applying
$\frac 1 {\nu !}\partial^\nu$ at $y=0$, for any $\nu\neq 0$, gives
\begin{align*}
-\bar a \Delta D^\alpha t_\nu & = \sum_{\beta\leq \alpha, \beta \neq 0} {\alpha\choose \beta} D^\beta \bar a \Delta D^{\alpha-\beta}Êt_\nu
+\sum_{\beta\leq \alpha} {\alpha\choose \beta}\nabla D^\beta \bar a \cdot \nabla D^{\alpha-\beta} t_\nu \\
&
+\sum_{\beta\leq \alpha} {\alpha\choose \beta}\sum_{j\in {\rm supp}(\nu)} (D^\beta\psi_j \Delta D^{\alpha-\beta}t_{\nu-e_j}+\nabla D^\beta\psi_j\cdot \nabla D^{\alpha-\beta}t_{\nu-e_j}).
\end{align*}
Integrating against $\Delta D^\alpha t_\nu$, and applying Young's inequality we find that, for any $\e >0$,
\begin{align*}
c_\nu^\alpha &\leq 
\e \sum_{\beta\leq \alpha, \beta \neq 0} {\alpha\choose \beta}\int_D  | D^\beta \bar  a| |\Delta D^\alpha t_{\nu}|^2+
\frac 1 {4\e}
\sum_{\beta\leq \alpha, \beta \neq 0} {\alpha\choose \beta} \int_D | D^\beta  \bar a| |\Delta D^{\alpha-\beta}t_{\nu}|^2 \\
& +\e
\sum_{\beta\leq \alpha} {\alpha\choose \beta} \int_D |\nabla D^\beta  \bar a| |\Delta D^\alpha t_{\nu}|^2
+\frac 1 {4\e}
\sum_{\beta\leq \alpha} {\alpha\choose \beta} \int_D |\nabla D^\beta  \bar a| |\nabla D^{\alpha-\beta}t_{\nu}|^2\\
& +\frac 1 2\sum_{j\in {\rm supp}(\nu)} \int_D |\psi_j| |\Delta D^\alpha t_{\nu}|^2+\frac 1 2\sum_{j\in {\rm supp}(\nu)} \int_D |\psi_j| |\Delta D^\alpha t_{\nu-e_j}|^2\\
&+ \e\sum_{\beta\leq \alpha, \beta \neq 0} {\alpha\choose \beta}\sum_{j\in {\rm supp}(\nu)} \int_D |D^\beta \psi_j| |\Delta D^\alpha t_{\nu}|^2
+ \frac 1 {4\e}\sum_{\beta\leq \alpha, \beta \neq 0} {\alpha\choose \beta}\sum_{j\in {\rm supp}(\nu)} \int_D |D^\beta \psi_j| |\Delta D^{\alpha-\beta} t_{\nu-e_j}|^2\\
& +\e
\sum_{\beta\leq \alpha} {\alpha\choose \beta}\sum_{j\in {\rm supp}(\nu)} \int_D |\nabla D^\beta \psi_j| |\Delta D^{\alpha}t_{\nu}|^2 +
\frac 1 {4\e}
\sum_{\beta\leq \alpha} {\alpha\choose \beta}\sum_{j\in {\rm supp}(\nu)} \int_D |\nabla D^\beta \psi_j| |\nabla D^{\alpha-\beta}t_{\nu-e_j}|^2\\
& = T_1+\cdots + T_{10}.
\end{align*}
After summation over $|\nu|=n$, the left side is $C^\alpha_n$ and the contribution of the odd numbered terms 
is bounded according to
\be
\sum_{|\nu|=n} (T_1+T_3+T_5+T_7+T_9) \leq \left(\frac \theta 2+B\e \right) C^\alpha_n,
\ee
where 
\be
B:=\Big \|\sum_{\beta\leq \alpha} {\alpha\choose \beta}(|\nabla D^\beta a|+\sum_{j\geq 1}|\nabla D^\beta \psi_j|)
+\sum_{\beta\leq \alpha, \beta \neq 0} {\alpha\choose \beta}(|D^\beta a|+\sum_{j\geq 1}|D^\beta \psi_j|)\Big \|_{L^\infty}<\infty.
\ee
The contribution of $T_6$ is bounded by
\be
\sum_{|\nu|=n} T_6 \leq \frac \theta 2 C^\alpha_{n-1}.
\ee
The contributions of the remaining terms can be bounded by
\be
\sum_{|\nu|=n} (T_2+T_4+T_8+T_{10}) \leq A \sum_{|\beta|\leq k-1} \Bigl( \sum_{|\nu|=n} \|D^\beta t_\nu\|_{L^2}^2+ \sum_{|\nu|=n-1} \|D^\beta t_\nu\|_{L^2}^2\Bigr),
\ee
where $A$ is a finite constant that depends on $\e$, $\bar a$, $\psi_j$, $k$ and $m$. From the induction assumption \iref{induct} we thus have
\be
\sum_{|\nu|=n} (T_2+T_4+T_8+T_{10}) \leq A M_{k-1} (\kappa_{k-1}^n+\kappa_{k-1}^{n-1}).
\ee
We take $\e>0$ small enough so that $\frac \theta 2+B\e<1$ and therefore obtain that
\be
C^\alpha_n \leq \kappa C^\alpha_{n-1}+\frac{ 2A M_{k-1}}{2-\theta-2B\e} (1+\kappa_{k-1}^{-1}) \kappa_{k-1}^n.
\quad \kappa:=\frac {\theta}{2-\theta-2B\e} <1.
\ee
By a similar argument as in the proof of Theorem \ref{theotaylorwH2}, this yields the existence of 
constants $M_k$ and $\kappa_{k-1}\leq \kappa_k<1$ such that
\be
C^\alpha_n\leq M_k\kappa_k^n, \quad n\geq 0,
\ee
which is the claimed estimate at order $k$. \hfill $\Box$

\subsection{Fractional and weighted regularity}

The summability results expressed by Theorem \ref{theotaylorHm} 
are derived for Sobolev spaces of integer order. In view of the induction argument
on $k$ used in the proof, we cannot treat fractional Sobolev spaces 
in a similar way. However, such spaces are well known to be obtainable
by interpolation of Sobolev spaces of integer order. For example,
when using the real interpolation method, for $0<\theta<1$ we have
\be
H^{1+\theta (k-1)}(D)=[H^1(D),H^k(D)]_{\theta,2},
\ee
as well as
\be
V\cap H^{1+\theta (k-1)}(D)=[V,V\cap H^k(D)]_{\theta,2}.
\ee
More generally, for any $0<\theta<1$ and $0<r\leq \infty$, we may consider
the real interpolation space
$[V,W^k]_{\theta,r}$, or the complex interpolation space
$[V,W^k]_{\theta}$. We refer to \cite{BL} for a general introduction to interpolation spaces.
If $Z$ is any of these spaces, we have the
interpolation inequality 
\be
\|t_\nu\|_{Z}\leq C\|t_\nu\|_{V}^{1-\theta} \|t_\nu\|_{W^k}^{\theta}.
\label{interineq}
\ee

The following result generalizes Theorem \ref{theotaylorHm} 
to these smoothness spaces.

\begin{cor}
\label{theotaylorfrac}
Assume that $\bar a\in L^\infty(D)$  is such that $\essinf \bar a>0$, and
that there exists a sequence $\rho=(\rho_j)_{j\geq 1}$ of positive numbers
such that \iref{uearho} holds.  Assume in addition that the right side $f$ in \iref{ellip}
belongs to $H^{k-2}(D)$, that the domain $D$ has $C^{k-2,1}$ smoothness,
and that $\bar a$ and all functions $\psi_j$
belong to $W^{k-1,\infty}(D)$. 
Assume that with a sequence $\bar \rho=(\bar \rho_j)_{j\geq 1}$ 
of positive numbers, 
\be
\sup_{|\alpha|\leq k-1} 
\Bigl\| \sum_{j\geq 1} \bar\rho_j |D^\alpha \psi_j| \Bigr\|_{L^\infty} 
<\infty\;.
\label{sumderrhomod}
\ee
Then for any $0<\theta<1$ and $0<r\leq \infty$ we have
\be
\sum_{\nu\in\cF} (\tilde\rho^\nu\|t_\nu\|_{Z})^2 <\infty,\quad   \tilde \rho^\nu:=\prod_{j\geq 1}  \tilde \rho_j^{\nu_j},
\ee
where $Z=[V,W^k]_{\theta,r}$ or $Z=[V,W^k]_{\theta}$, 
and $\tilde \rho_j=\rho_j^{1-\theta}\bar \rho_j^\theta$.
\end{cor}

\begin{proof}
By our assumptions, we have
\be
\sum_{\nu\in\cF} (\rho^{\nu} \|t_\nu\|_{V})^2 <\infty \quad\text{and}\quad 
\sum_{\nu\in\cF} (\bar \rho^{\nu} \|t_\nu\|_{W^k})^2 <\infty.
\ee
Using the interpolation inequality \iref{interineq}, we obtain
\be
\sum_{\nu\in\cF} (\tilde\rho^\nu\|t_\nu\|_{Z})^2 \leq C^2\sum_{\nu\in\cF} (\rho^\nu\|t_\nu\|_{V})^{2(1-\theta)} (\bar\rho^\nu\|t_\nu\|_{W^k})^{2\theta}
<\infty,
\ee
by H\"older's inequality.
\end{proof}

Note that Corollary \ref{theotaylorfrac} can be used to derive $(\norm{t_\nu}_Z)_{\nu\in\cF} \in \ell^p(\cF)$ for some $p\in]0,2[$ even in cases where \eqref{sumderrhomod} requires $\bar\rho_j\to 0$, as long as the resulting interpolated weights $\tilde\rho_j$ still satisfy $\tilde \rho_j>1$ and $(\tilde\rho_j^{-1})_{j\geq 1} \in \ell^q(\N)$ for some $q>0$. We apply the result in this manner in \S\ref{sec:sobolevregexample}.

As already noted, the spaces $W^k$ 
coincide with the Sobolev space $V\cap H^k$ only when the domain $D$
has $C^{k-1,1}$-smooth boundary. In addition, Theorem \ref{theotaylorwHm}
requires that $D$ has at least $C^{k-2,1}$-smooth boundary. 
One way to circumvent this limitation is to search for 
analogous results when the spaces 
$W^k$ are replaced by suitable \emph{weighted Sobolev spaces}.
Such spaces are particularly relevant to the case where $D$ is a
polygon or polyhedron, possibly with re-entreant corners.
We discuss the adaptation of Theorems  
\ref{theotaylorwH2} and \ref{theotaylorwHm} to this situation.

To simplify the exposition, we confine the discussion to space dimension $m=2$
and shall assume in the following 
that $D\subset \R^2$ is a bounded polygon with $J$  vertices.
For $j=1,\dots,J$, we denote by $\bc_j$ these vertices
and define for any $x\in D$ the truncated distances
\be
r_j(x):= 1\wedge {\rm dist}(x,\bc_j).
\ee
For $\beta\in\R$, we define
\be
\Phi_\beta(x) := \prod_{j\geq 1} r_j(x)^{\beta}.
\ee
For any integer $k\geq 0$ and $\theta \in \R$, 
we define the Kondrat'ev spaces
\be
K^k_\theta(D) 
:= \{ u \; : \;  \Phi_{|\alpha|-\theta} \partial_x^\alpha u \in L^2(D):\;
|\alpha| \leq k \}.
\ee
For $k=0$, we also write $L^2_\theta(D) =K^0_\theta(D)$.
These weighted spaces are relevant to us for two reasons. First, 
the Dirichlet problem for the Poisson equation in a polygon $D$
admits a shift theorem in these weighted spaces: for all $k\geq 2$, 
there exists $\eta>0$ depending on $D$ and a constant $C$ depending on $D$ and $k$,
such that, for $|\theta|<\eta$,
\begin{equation}
\label{eq:Kshift}
\| u \|_{K^k_{\theta+1}} \leq C \| \Delta u\|_{K^{k-2}_{\theta-1}}\;.
\end{equation}
Second, the approximation property \iref{spatialappn}
remains valid for the space $X=K^k_{\theta+1}(D)$ 
with the same rate $t=\frac {k-1}2$
as for $X=H^k(D)$, if we use Lagrange finite element spaces $(V_n)_{n>0}$ 
of degree at least $k-1$ and with appropriate mesh refinement near the corners of $D$
(see \cite{BNZ2005,GaspMorin09}, \cite[Eqn. (0.3)]{AN2015} and the references there).
In the case $k=2$, it is known that $W\subset K^2_{\theta+1}(D)$ for $|\theta|$ small enough.
Specifically, for the homogeneous Dirichlet boundary conditions presently considered,
when $D$ is a polygon with straight sides, $|\theta| < \pi/\omega_{max}$ where
$\omega_{max} \in ]0,2\pi]$ denotes the maximal interior opening angle at the corners
of $D$.
Therefore, the same assumptions as those in Theorem \ref{theotaylorH2} imply the result
\be
\sum_{\nu\in\cF} (\rho^{\nu} \|t_\nu\|_{K^2_{\theta+1}})^2 <\infty, 
\label{taylorwK2}
\ee
in the case where $D$ is a polygon. 
One can adapt the proof of Theorem \ref{theotaylorH2} in order to show that \iref{taylorwK2}
holds under the assumption
\be
 \left\| \sum_{j\geq 1} \rho_j \Phi_\gamma|\nabla \psi_j| \right\|_{L^\infty} <\infty,
\ee
where $\gamma=\max\{1-\theta,0\}$. 
For $| \theta | < 1$ 
(e.g., for polygonal domains $D$ with reentrant corners),
this assumption is slightly weaker than \iref{sumgradrho} due to the presence of 
weight functions $\Phi_\gamma$ which vanish at the corners of the domain, 
thereby allowing the $\psi_j$ to be singular at these points.
For $k>2$, we conjecture that, under similar assumptions, Theorem \ref{theotaylorwHm}
extends to a polygonal domain with $W^k$ replaced by $K^k_{\theta+1}(D)$
for $|\theta|$ small enough.
\section{Jacobi expansions}
We now prove Theorem \ref{theojacobiHm}, transferring the results for Taylor expansions 
in the previous sections to Jacobi series. 
The corresponding univariate Jacobi polynomials $(J_k)_{k\geq 0}$ 
are orthonormal with respect to the probability measure $d_{\alpha,\beta}(t)\,dt$ 
on $[-1,1]$ with $\alpha,\beta >-1$, where
\begin{equation}
d_{\alpha,\beta}(t) = \frac{\Gamma(\alpha + \beta + 2)}{2^{\alpha + \beta + 1}  
\Gamma(\alpha+1) \Gamma(\beta+1) } (1- t)^{\alpha} (1+ t)^{\beta}.
\end{equation}
For the corresponding orthonormal polynomials, one has the Rodrigues' formula
\begin{equation}\label{jacobirodrigues}
	J_k(t) = \frac{c^{\alpha,\beta}_{k} }{ k! \, 2^{k} } 
     (1 - t)^{-\alpha} ( 1+ t)^{-\beta} \,\frac{d^{k}}{d t^k} 
       \bigl( (t^2 - 1)^{k} (1-t)^{\alpha} (1+t)^{\beta} \bigr),
\end{equation}
where $c^{\alpha,\beta}_0 = 1$ and
\begin{equation}\label{normalizationfactor}
	   c^{\alpha, \beta}_{k}  = \sqrt{ \frac{ (2k + \alpha + \beta + 1)\, k!\, \Gamma(k + \alpha + \beta + 1) \,\Gamma(\alpha+1)\, \Gamma(\beta+1) }{ \Gamma(k + \alpha + 1)\, \Gamma(k + \beta + 1) \, \Gamma(\alpha + \beta + 2)} } ,\quad k \in\N.
\end{equation}
Notable special cases are Legendre polynomials for the uniform measure, where $\alpha=\beta=0$, and Chebyshev polynomials, where $\alpha=\beta=-\frac12$. 
On $U$, we consider the product measure
\begin{equation}
  d\sigma(y)  = \bigotimes_{j\geq 1} d_{\alpha,\beta}(y_j)\,dy_j,
\end{equation}
and recall that the tensor product polynomials $J_\nu(y) = \prod_{j\geq 1} J_{\nu_j}(y_j)$, $\nu\in\cF$, are an orthonormal basis of $L^2(U,\sigma)$.
Our aim is now to quantify the summability of $W^k$ norms of Jacobi coefficients
\begin{equation}
    v_\nu = \int_U u(y)\, J_\nu(y)\,d\sigma(y), \quad \nu \in \cF,
\end{equation}
as stated in Theorem \ref{theojacobiHm}. This assertion on $\ell^p$ summability will again be derived from a result on weighted $\ell^2$ summability.

\begin{theorem}\label{jacobiweightedl2}
	Under the assumptions of Theorem \ref{theotaylorH2}
when $k=2$ and of Theorem \ref{theotaylorHm} when $k>2$,
	\begin{equation}
		\sum_{\nu\in\cF} \bigl( a_\nu^{-1} \rho^\nu \norm{v_\nu}_{W^k} \bigr)^2 < \infty, \qquad a_\nu := \prod_{j\geq 1} c^{\alpha,\beta}_{\nu_j}.
	\end{equation}
\end{theorem}

\begin{proof}
  Closely following the proof of \cite[Thm.\ 3.1]{BCM}, for $y, z \in U$, we set $T_y z := \bigl( y_j +  ( 1- \abs{y_j}) \rho_j z_j \bigr)_{j\geq 1}$ and $w_y(z) := u(T_y z)$. 
For each $y\in U$, let $\bar a_y :=a(y)= \bar a + \sum_{j\geq 1} y_j \psi_j$ and $\psi_{y,j} := (1 - \abs{y_j}) \rho_j \psi_j$. Then $w_y$ solves the modified affine-parametric problem
\begin{equation}\label{pdewy}
   \int_D \Bigl(  \bar a_y  + \sum_{j\geq 1} z_j \psi_{y,j} \Bigr) \nabla w_y(z) \cdot \nabla v\,dx =  \int_D f \, v\,dx , \quad v \in V,
\end{equation}
with Taylor coefficients
\begin{equation}\label{jacobitayloridentity}
   t_{y,\nu} := \frac1{\nu!} \partial^\nu w_y(0) = \frac1{\nu!}  \Bigl( \prod_{j \geq 1}  (1 - \abs{y_j})^{\nu_j}  \Bigr)  \rho^\nu \partial^\nu u(y) \,.
\end{equation}
We have the following $y$-uniform bounds: on the one hand,
\begin{equation}\label{jacobiunif1}
  \norm{\bar a_y^{-1}}_{L^\infty(D)}  \leq  \Bignorm{ \Bigl( \bar{a} - \sum_{j\geq 1} | \psi_j|  \Bigr)^{-1} }_{L^\infty(D)} < \infty ,
\end{equation}
as well as
\begin{equation}\label{jacobiunif2}
   \norm{ D^\alpha \bar a_y}_{L^\infty(D)}	\leq  \Bignorm{ \abs{ D^\alpha \bar{a} } + \sum_{j\geq 1} \abs{D^\alpha \psi_j }  }_{L^\infty(D)} < \infty, \quad \abs{\alpha} \leq k - 1;
\end{equation}
on the other hand, 
\begin{equation}\label{jacobiunif3}
\begin{aligned}
 \Biggl\| \frac{\sum_{j\geq 1} \abs{ \psi_{y,j} }}{\bar a_y} \Biggr\|_{L^\infty(D)} 
  &  \leq \Biggl\|  \frac{\sum_{j\geq 1} \rho_j \abs{ \psi_{j} } - \sum_{j\geq 1} \rho_j \abs{y_j} \abs{\psi_j} }{\bar a - \sum_{j\geq 1} \abs{y_j} \abs{\psi_j} } \Biggr\|_{L^\infty(D)}  \\
  &  \leq  \Biggl\|  \frac{\sum_{j\geq 1} \rho_j \abs{ \psi_{j} }  }{\bar a } \Biggr\|_{L^\infty(D)} < 1, 
\end{aligned}
\end{equation}
and
\begin{equation}\label{jacobiunif4}
  \Bignorm{ \sum_{j\geq 1} \abs{ D^\alpha \psi_{y,j} } }_{L^\infty(D)} \leq \Bignorm{ \sum_{j\geq 1} \rho_j \abs{ D^\alpha \psi_j } }_{L^\infty(D)} < \infty , \quad \abs{\alpha}\leq k-1.
\end{equation}
By Theorem \ref{theotaylorwH2} for $k=2$ and Theorem \ref{theotaylorwHm} for $k>2$, we thus obtain
\begin{equation}
	\sum_{\nu\in\cF}  \Bigl\| \frac{1}{\nu!}\partial^\nu w_y(0) \Bigr\|_{W^k}^2 = \sum_{\nu\in\cF} \|t_{y,\nu}\|_{W^k}^2  \leq C < \infty,
	\label{uniy}
\end{equation}
with $C>0$ independent of $y$.
Moreover, by \eqref{jacobirodrigues} and integration by parts,
	\begin{equation}
	 v_\nu = \int_U u(y)\, J_\nu(y)\,d\sigma(y) =  a_\nu \int_U \frac{1}{\nu!} \partial^\nu u(y) \prod_{j\geq 1} \frac{(1-y_j^2)^{\nu_j}}{2^{\nu_j}} d\sigma(y),
	 \end{equation}
and consequently	 
\begin{equation}
\begin{aligned}
  \sum_{\nu\in\cF} a_\nu^{-2} \rho^{2\nu} \| v_\nu\|_{W^k}^2 
   & \leq \int_U \sum_{\nu\in\cF} \rho^{2\nu} \Bigl\| \frac{1}{\nu!}\partial^\nu u(y) \Bigr\|_{W^k}^2 \prod_{j\geq 1} (1 - |y_j|)^{2\nu_j}\frac{ ( 1 + |y_j|)^{2\nu_j} }{2^{2\nu_j}}   \,d\sigma(y)  \\  
   &  \leq  \int_U \sum_{\nu\in\cF}  \Bigl\| \frac{1}{\nu!}\partial^\nu w_y(0) \Bigr\|_{W^k}^2 \,d\sigma(y)   \leq   C,\end{aligned} 
\end{equation}
with $C$ as in \iref{uniy}.
\end{proof}

\begin{proof}[{\rm \textbf{Proof of Theorem \ref{theojacobiHm}}}]
	From Theorem \ref{jacobiweightedl2}, we obtain the $\ell^p$ summability of $(\norm{v_\nu}_{W^k})_{\nu\in\cF}$ using a H\"older estimate 
	analogous to \iref{holder} which was used in the proof of Theorems \ref{theotaylorH2} and \ref{theotaylorHm}, but this time using 
\begin{equation}\label{jacobihoelder}
		\biggl(\prod_{j\geq1} a_\nu \rho_j^{-\nu_j} \biggr)_{\nu\in\cF} \in \ell^q(\cF).
\end{equation}
To see that this holds true under the given assumptions, note first that for some $c,t>0$ depending on $\alpha,\beta$, we have $c^{\alpha,\beta}_k \leq 1 + ck^t$ for $k \in \N_0$, and 
\begin{equation}\label{jacobihoelderfactorization}
	\sum_{\nu\in\cF} \rho^{- q \nu} \prod_{j\geq1} (1 + c\nu_j^t)^{q} 
	= \prod_{j\geq 1} \Bigl( \sum_{k=0}^\infty \rho_j^{-qk} (1 + ck^t)^{q} \Bigr)
	\leq \prod_{j\geq 1} ( 1 + C \rho_j^{-q}),
\end{equation}
where in the last step we have used that $\rho_{\min}:=\min_{j\geq 1} \rho_j > 1$, with $C>0$ depending on $c$, $t$, $q$, and $\rho_{\min}$. The infinite product on the right converges precisely when $(\rho_j^{-1})_{j\geq 1} \in \ell^q(\N)$, and we thus have \eqref{jacobihoelder}.
\end{proof}

\begin{remark}
	The above arguments also apply to tensor product Jacobi measures with variable parameter sequences $(\alpha_j)_{j\geq 1}$, $(\beta_j)_{j\geq 1}$ with $\alpha_j,\beta_j >-1$, that is,
	\begin{equation}
	  d\sigma(y) = \bigotimes_{j\geq 1} d_{\alpha_j,\beta_j}(y_j)\,dy_j,
	\end{equation}
	provided that there exists a fixed polynomial $Q$ such that $c^{\alpha_j,\beta_j}_k \leq Q(k)$, $k\in\N_0$. Since $Q$ can be chosen to satisfy $Q(0)=1$ without loss (up to increasing the degree by one), the estimate \eqref{jacobihoelderfactorization} remains valid also in this case.
\end{remark}

\begin{remark}\label{fracjacobi}
The summability results established in \S 4.3  for the fractional and weighted Sobolev norms
of Taylor coefficients also easily extend to the Jacobi coefficients by similar arguments. In particular, under the assumptions of Theorem \ref{theotaylorfrac}, we also obtain
\begin{equation}	\label{fracjacobibound}
\sum_{\nu\in\cF} \bigl(a_\nu^{-1}\tilde\rho^\nu\|v_\nu\|_{Z}\bigr)^2 <\infty
\end{equation}
with $\tilde\rho_j$ as defined there.
\end{remark}

\section{Hermite expansions}

We now turn to lognormal diffusion coefficients of the form $a=\exp(b)$
and the proof of Theorem \ref{theohermiteH2}. This proof is more involved than
those of the previous results, but follows in part a similar route as the proof of
Theorem \ref{theohermiteV} from \cite{BCDM}. For this reason, we only sketch
below the arguments when they are the same, and we detail the part that differs
significantly.

By similar arguments as in \cite[\S 2]{BCDM}, one shows that under the assumptions of Theorem \ref{theohermiteH2}, one has 
for almost every $y\in U=\R^\N$ in the sense of the Gaussian product measure $\gamma$,
\begin{equation}
\norm{b(y)}_{L^\infty}<\infty \quad {\rm and}Ê\quad \|\nabla b(y)\|_{L^\infty}  < \infty, 
\end{equation}
where we have used the notation $\|\nabla \psi\|_{L^\infty}:= \| |\nabla \psi| \|_{L^\infty}$, 
 from which one concludes $u(y) \in W$ for such $y$. In addition, one has the moment bounds
\begin{equation}
\label{boundedmoments}
	\E(\exp(k\norm{b}_{L^\infty})) , \; \E(\exp(k\|\nabla b\|_{L^\infty})), \; \E(\norm{u}_V^k), \; 
\E(\norm{u}_W^k) \; < \;\infty
\end{equation}
for all $0\leq k < \infty$. 
Moreover, following the lines of \cite[Theorem 3.3]{BCDM}, we also obtain
\begin{equation}
\label{sumidentity}
\sum_{\nu\in\cF} b_\nu \norm{u_\nu}^2_W 
= 
\sum_{\norm{\mu}_{\ell^\infty} \leq r} \frac{\rho^{2\mu}}{\mu!} \int_U \norm{\partial^\mu u(y)}_W^2 \,d\gamma(y),
\end{equation}
where
\begin{equation}
b_\nu:=\sum_{\|\mu\|_{\ell^\infty}\leq r}{\nu\choose \mu} \rho^{2\mu}.
\end{equation}
Here, we recall from \cite{BCDM} the notation
\begin{equation}
{\nu\choose \mu}:=\prod_{j\geq 1}{\nu_j \choose \mu_j},\quad \mu,\nu\in\cF
\end{equation}
with the convention
\begin{equation}
{n\choose m}:=0, \quad \text{if $m>n$.}
\end{equation}
The central estimate for the proof of Theorem  \ref{theohermiteH2} is the following.

\begin{theorem}
\label{hermiteapproxauxiliary}
	Let $r\geq 1$ be an integer, and let $(\rho_j)_{j\geq 1}$ be a positive sequence such that
	\begin{equation}
		K := \left\| \sum_{j\geq 1} \rho_j \abs{\psi_j}\right\|_{L^\infty(D)} 
                   < \frac{\ln \theta}{\sqrt{r}}, \qquad \theta:= 1 + \biggl ( 1 - \frac1{\sqrt{2}}\biggr)^2,
\label{Kcondition}\end{equation}
and
	\begin{equation}
	\hat K := \left\| \sum_{j\geq 1} \rho_j \abs{\nabla\psi_j} \right\|_{L^\infty(D)} < \infty\;.
	\end{equation}
	Then
	\begin{equation}
	\label{yintegralestimate}
	\sum_{\norm{\mu}_{\ell^\infty} \leq r} \frac{\rho^{2\mu}}{\mu!} 
         \int_U \norm{\partial^\mu u(y)}_W^2 \,d\gamma(y) < \infty.
	\end{equation}
\end{theorem}

\begin{proof}
We first establish a bound pointwise a.e.\ in $y$. 
Let thus $y\in U$ with $\norm{b(y)}_{L^\infty}<\infty$ and 
$\|\nabla b(y)\|_{L^\infty} < \infty$ be fixed.
Since $a(y)\in W^{1,\infty}(D)$ and $u(y)\in W$, we have 
\begin{equation}
\label{strongformy}
a(y) \Delta u(y) = - f - \nabla a(y)\cdot \nabla u(y)	
\end{equation}
in $L^2(D)$. 
Similarly to \cite[Lemma 3.1]{BCDM}, using the notation
\begin{equation}
	S_\mu := \{\nu \in \cF: \nu \le \mu \ \text{and} \ \nu \neq \mu\},
	\quad \mu \in \cF,
\end{equation}
for $\mu \neq 0$ we obtain
\begin{equation}
\label{Laplacianderiv}
	a(y) \Delta \partial^\mu u(y) = - \sum_{\nu\in S_\mu} {\mu\choose \nu} \partial^{\mu-\nu} a(y)\, \Delta \partial^\nu u(y) - \sum_{0\leq \nu\leq \mu} {\mu\choose \nu} \nabla \partial^{\mu - \nu} a(y) \cdot \nabla \partial^\nu u(y).
\end{equation}
We now establish bounds for 
\begin{equation}
	\sigma_k := \sum_{\mu\in\Lambda_k} \frac{\rho^{2\mu}}{\mu!} \int_D a(y)\abs{\nabla \partial^\mu u(y)}^2\,dx, 
	\quad
	\hat\sigma_k := \sum_{\mu\in\Lambda_k} \frac{\rho^{2\mu}}{\mu!} \int_D a(y)\abs{\Delta \partial^\mu u(y)}^2\,dx,
\end{equation}
where
\begin{equation}
	\Lambda_k := \{ \mu\in\cF \colon \abs{\mu} = k, \norm{\mu}_{\ell^\infty} \leq r \}.
\end{equation}
Under the given assumptions, \cite[Theorem 4.1]{BCDM} yields
\begin{equation}
\label{sigmaest}
\sigma_k \leq \delta^k \sigma_0\quad \text{with $\delta := \sqrt{r} K / \ln \theta< 1$},
	\end{equation}
and it remains to bound $\hat\sigma_k$. Note first that
$\partial^{\mu-\nu} a(y) = a(y) \psi^{\mu-\nu}$ and
\begin{equation}
	\nabla \bigl(  a(y) \psi^{\mu-\nu} \bigr)
	 = a(y) \psi^{\mu-\nu} \nabla b(y) + a(y) \sum_{j\in\supp(\mu-\nu)} (\mu_j-\nu_j) \psi^{\mu-\nu-e_j} \nabla\psi_j.
\end{equation}
Here and further we use the notation 
\be
\psi^\nu:=\prod_{j\geq 1} \psi_j^{\nu_j} \quad {\rm and}\quad |\psi |^\nu:=\prod_{j\geq 1} |\psi_j|^{\nu_j}.
\ee
Thus multiplication of \eqref{Laplacianderiv} by $\rho^{2\mu} \Delta\partial^\mu u(y)/ \mu!$, integration over $D$, and summation over $\mu\in\Lambda_k$ yield
\begin{equation}
	\hat\sigma_k \leq S_1 + S_2 + S_3 + S_4,
\end{equation}
where
\begin{align}
	 S_1 &:= \sum_{\mu\in\Lambda_k} \sum_{\nu\in S_\mu} \int_D \frac{\rho^{2\mu}\abs{\psi}^{\mu-\nu}}{(\mu-\nu)! \nu!} a(y) \abs{\Delta \partial^\nu u(y)}\abs{\Delta\partial^\mu u(y)}\,dx ,\\
	 S_2 &:= \sum_{\mu\in\Lambda_k} \frac{\rho^{2\mu}}{\mu!} \int_D a(y)\abs{\nabla b(y)} \abs{\nabla \partial^\mu u(y)}\abs{\Delta \partial^\mu u(y)}\,dx, \\
	 S_3 &:= \sum_{\mu\in\Lambda_k} \sum_{\nu\in S_\mu} \int_D \frac{\rho^{2\mu}\abs{\psi}^{\mu-\nu}}{(\mu-\nu)! \nu!} a(y) \abs{\nabla b(y)} \abs{\nabla \partial^\nu u(y)}\abs{\Delta\partial^\mu u(y)}\,dx ,\\
	 S_4 &:= \sum_{\mu\in\Lambda_k} \sum_{\nu\in S_\mu} \sum_{j\in\supp(\mu-\nu)} \int_D \frac{\rho^{2\mu} \abs{\psi}^{\mu-\nu-e_j} }{(\mu-\nu-e_j)! \nu!} a(y) \abs{\nabla\psi_j}\abs{\nabla\partial^\nu u(y)} \abs{\Delta\partial^\mu u(y)}\,dx.
\end{align}
Introducing the functions
\begin{equation}
	\varepsilon(\mu,\nu) = \sqrt{\frac{\mu!}{\nu!}} \frac{\rho^{\mu-\nu} \abs{\psi}^{\mu-\nu}}{(\mu-\nu)!},
\end{equation}
we can proceed exactly as in the proof of \cite[Theorem 4.1]{BCDM} to show
\begin{equation}
\label{S1bound}
	S_1 = \sum_{\mu\in\Lambda_k} \sum_{\nu\in S_\mu} \int_D \varepsilon(\mu,\nu) a(y) \frac{\rho^\nu \abs{\Delta\partial^\nu u(y)}}{\sqrt{\nu!}} \frac{\rho^\mu\abs{\Delta\partial^\mu u(y)}}{\sqrt{\mu!}} \,dx 
	 \leq \biggl( \sum_{\ell = 0}^{k-1} \frac{(\sqrt{r} K)^{k-\ell}}{(k-\ell)!} \hat\sigma_\ell\biggr)^{\frac12} \hat\sigma_k^{\frac12}.
\end{equation}
For $S_2$, by the Cauchy-Schwarz inequality we immediately obtain
\begin{equation}
	S_2 \leq \|\nabla b(y)\|_{L^\infty} \sqrt{\sigma_k \hat\sigma_k}.
\end{equation}
Proceeding again as in \cite[Theorem 4.1]{BCDM} also gives
\be
\begin{aligned}
	S_3 &= \sum_{\mu\in\Lambda_k} \sum_{\nu\in S_\mu} \int_D \varepsilon(\mu,\nu) a(y) \abs{\nabla b(y)}\frac{\rho^\nu \abs{\nabla\partial^\nu u(y)}}{\sqrt{\nu!}} \frac{\rho^\mu\abs{\Delta\partial^\mu u(y)}}{\sqrt{\mu!}} \,dx \\
	&\leq \|\nabla b(y)\|_{L^\infty} \biggl( \sum_{\ell = 0}^{k-1} \frac{(\sqrt{r} K)^{k-\ell}}{(k-\ell)!} \sigma_\ell\biggr)^{\frac12} \hat\sigma_k^{\frac12}.
\end{aligned}
\ee
By \eqref{sigmaest}, 
\begin{equation}
	\sum_{\ell = 0}^{k-1} \frac{(\sqrt{r} K)^{k-\ell}}{(k-\ell)!} \sigma_\ell  <  \delta^k \sigma_0,
\end{equation}
and consequently
\begin{equation}
\label{S23bound}
	S_2 + S_3 < 2 \|\nabla b(y)\|_{L^\infty} \delta^{\frac{k}2} (\sigma_0 \hat\sigma_k)^{\frac12}.
\end{equation}
In order to bound $S_4$, we introduce
\begin{equation}
	\hat\varepsilon(\mu,\nu) := \sum_{j\in\supp(\mu-\nu)} \sqrt{\frac{\mu!}{\nu!}} \frac{\rho^{\mu-\nu-e_j} \abs{\psi}^{\mu-\nu-e_j}}{(\mu-\nu-e_j)!} \rho_j \abs{\nabla \psi_j},
\end{equation}
and thus obtain
\begin{equation}
	S_4 = \int_D \sum_{\mu\in\Lambda_k} \sum_{\nu\in S_\mu} \hat\varepsilon(\mu,\nu) a(y) \frac{\rho^\nu \abs{\nabla \partial^\nu u(y)}}{\sqrt{\nu!}} \frac{\rho^\mu \abs{\Delta \partial^\mu u(y)}}{\sqrt{\mu!}}\,dx .
\end{equation}
By Cauchy-Schwarz on the summations over $\mu$ and $\nu$, the expression on the right is bounded from above by
\begin{equation}
\label{S4est}
	\int_D \biggl( \sum_{\mu\in\Lambda_k} \sum_{\nu\in S_\mu} \hat\varepsilon(\mu,\nu) a(y) \frac{\abs{\rho^\nu \nabla \partial^\nu u(y)}^2}{\nu!} \biggr)^{\frac12} 
	\biggl( \sum_{\mu\in\Lambda_k}  a(y) \frac{\abs{\rho^\mu \Delta \partial^\mu u(y)}^2}{\mu!} \sum_{\nu\in S_\mu} \hat\varepsilon(\mu,\nu) \biggr)^{\frac12}\,dx.
\end{equation}
On the one hand, using that $\mu!/\nu! \leq r^{\abs{\mu-\nu}}$, we estimate
\begin{equation}
\label{hateps1}
\begin{aligned}
  	\sum_{\nu\in S_\mu} \hat\varepsilon(\mu,\nu) &\leq
  	  \sum_{\ell=1}^k r^{\ell/2} \sum_{\substack{\nu\in S_\mu \\ \abs{\mu-\nu}=\ell}} \sum_{j \in \supp(\mu - \nu)} \rho_j \abs{\nabla \psi_j} \frac{\rho^{\mu-\nu-e_j} \abs{\psi}^{\mu-\nu-e_j}}{(\mu-\nu-e_j)!} \\
  	  &\leq \sum_{\ell=1}^k r^{\ell/2} \sum_{j\geq 1} \rho_j \abs{\nabla \psi_j} \sum_{\abs{\tau}=\ell-1} \frac{\rho^\tau \abs{\psi}^\tau}{\tau!} 
  	  \leq \hat K \sum_{\ell = 1}^k r^{\ell/2} \frac{K^{\ell-1}}{(\ell-1)!}
  	   \leq \sqrt{r} \hat K e^{\sqrt{r} K}.
\end{aligned}
\end{equation}
On the other hand, with $R_{\nu,k}:= \{ \mu\in\Lambda_k\colon \nu\in S_\mu\}$, we have
\begin{equation}
	\sum_{\mu\in\Lambda_k} \sum_{\nu\in S_\mu} \hat\varepsilon(\mu,\nu) a(y) \frac{\abs{\rho^\nu \nabla \partial^\nu u(y)}^2}{\nu!}
	  = \sum_{\ell=0}^{k-1} \sum_{\nu \in \Lambda_\ell} a(y) \frac{\abs{\rho^\nu \nabla \partial^\nu u(y)}^2}{\nu!} \sum_{\mu \in R_{\nu,k}} \hat\varepsilon(\mu,\nu)
\end{equation}
where, for each $\nu\in\Lambda_\ell$,
\begin{equation}
\label{hateps2}
\begin{aligned}
  	\sum_{\mu\in R_{\nu,k}} \hat\varepsilon(\mu,\nu) 
  	  &\leq  \sum_{\mu\in R_{\nu,k}}  r^{\abs{\mu-\nu}/2} \sum_{j \in \supp(\mu - \nu)} \rho_j \abs{\nabla \psi_j} \frac{\rho^{\mu-\nu-e_j} \abs{\psi}^{\mu-\nu-e_j}}{(\mu-\nu-e_j)!} \\
  	  & \leq r^{(k-\ell)/2} \sum_{j\geq 1} \rho_j \abs{\nabla \psi_j} \sum_{\abs{\tau}=k -\ell-1}	\frac{\rho^\tau \abs{\psi}^\tau}{\tau!} \leq r^{(k-\ell)/2} \hat K \frac{K^{k-\ell-1}}{(k-\ell-1)!}.
\end{aligned}
\end{equation}
Combining \eqref{S4est} with \eqref{hateps1} and \eqref{hateps2} and using Cauchy-Schwarz for the integral over $D$ yields
\begin{equation}
	S_4 \leq \sqrt{r} \hat K e^{\frac12\sqrt{r} K} \biggl( \sum_{\ell=0}^{k-1} \frac{(\sqrt{r} K)^{k-\ell-1}}{(k-\ell-1)!} \sigma_\ell \biggr)^{\frac12} \hat\sigma_k^{\frac12} \leq \sqrt{\frac{\theta r}{\delta}} \hat K e^{\frac12 \sqrt{r} K} \sigma_0^{\frac12} \delta^{\frac{k}2} \hat\sigma_k^{\frac12}, 
\end{equation}
where we have used that by \eqref{sigmaest},
\begin{equation}
\label{S4bound}
	\sum_{\ell=0}^{k-1} \frac{(\sqrt{r} K)^{k-\ell-1}}{(k-\ell-1)!} \sigma_\ell 
	\leq \delta^{k-1} \sum_{\ell=0}^{k-1} \frac{(\ln \theta)^{k-\ell-1}}{(k-\ell-1)!} \sigma_0 \leq \delta^{k-1} \theta \sigma_0.
\end{equation}
In summary, from \eqref{S1bound}, \eqref{S23bound}, and \eqref{S4bound} we obtain the recursive estimate
\begin{equation}
	\label{sigmahatbound}
	\hat\sigma_k \leq
	\biggl( \sum_{\ell = 0}^{k-1} \frac{(\sqrt{r} K)^{k-\ell}}{(k-\ell)!} \hat\sigma_\ell\biggr)^{\frac12} \hat\sigma_k^{\frac12}
	+ C_0^{\frac12}  \delta^{\frac{k}2}  \hat\sigma_k^{\frac12},
\end{equation}
where
\begin{equation}
	C_0 :=  \Bigl( 2 \|\nabla b(y)\|_{L^\infty}  + \sqrt{\frac{\theta r}{\delta}} \hat K e^{\frac12 \sqrt{r} K} \Bigr)^2 \sigma_0.
\end{equation}

We now prove by induction that
\begin{equation}
\label{hatsigmadecay}
	\hat\sigma_k \leq 2 \hat C_0 \delta^k, \quad\hat C_0 := \max\{ \hat\sigma_0 , C_0 \}.
\end{equation}
This evidently holds true for $k=0$. 
Assuming that $\hat\sigma_\ell \leq 2 \hat C_0 \delta^\ell$ for $0\leq \ell < k$, 
either $\hat\sigma_k = 0$, in which case there is nothing to prove, or
\begin{equation}
	\hat\sigma_k^{\frac12} \leq \Biggl( \sqrt{2} \biggl(\sum_{\ell=0}^{k-1} \frac{(\ln\theta)^{k-\ell}}{(k-\ell)!} \biggr)^{\frac12} + 1 \Biggr) \hat C_0^{\frac12} \delta^{\frac{k}2} .
\end{equation}
Since
\begin{equation}
	\biggl(\sum_{\ell=0}^{k-1} \frac{(\ln\theta)^{k-\ell}}{(k-\ell)!} \biggr)^{\frac12} \leq \sqrt{2} (e^{\ln\theta} - 1)^{\frac12} = \sqrt{2}\biggl( 1 - \frac1{\sqrt{2}}\biggr) = \sqrt{2}-1,
\end{equation}
this shows \eqref{hatsigmadecay}.
In summary, as a consequence of \eqref{hatsigmadecay} we have, by summation over $k$,
\begin{multline}
\label{aeyestimate}
\sum_{\norm{\mu}_{\ell^\infty}\leq r} \frac{\rho^{2\mu}}{\mu!} 
\int_D a(y)\abs{\Delta \partial^\mu u(y)}^2\,dx \\
\leq 
	C \max\biggl\{ \int_D a(y) \abs{\Delta u(y)}^2 \,dx, \bigl( 1 + \|\nabla b(y)\|_{L^\infty}  \bigr)^2 \int_D a(y) \abs{\nabla u(y)}^2\,dx \biggr\},
\end{multline}
where $C>0$ depends on $r$, $K$, $\hat K$, $\delta$, but not on $y$. 
The estimate \eqref{yintegralestimate} now follows with
\begin{equation}
\norm{\partial^\mu u(y)}^2_W 
\leq \exp(\norm{b(y)}_{L^\infty(D)}) \int_D a(y) \abs{\Delta \partial^\mu u(y)}^2 \,dx
\end{equation}
and the boundedness of moments \eqref{boundedmoments}, similarly to \cite[Theorem 4.2]{BCDM}.
\end{proof}

Note that under the assumptions of Theorem \ref{theohermiteH2}, any $(\rho_j)_{j\geq 1}$ such that \eqref{condhermite} holds can be rescaled, without affecting the $\ell^q$ summability of $(\rho_j^{-1})_{j\geq 1}$, such that the condition \eqref{Kcondition} in Theorem \ref{hermiteapproxauxiliary} holds. Using Theorem \ref{hermiteapproxauxiliary} and \eqref{sumidentity}, we can thus conclude the proof of Theorem \ref{theohermiteH2},
by means of a H\"older inequality, exactly as in \cite[Section 5]{BCDM}, 
choosing $r$ depending on $p$.
\section{Towards space-parameter adaptivity}
\label{sec:Adap}

The space-parameter approximation results presented 
in \S \ref{sec:SpcParDiscr} are based on analyzing the error resulting from
space discretization of each coefficient $u_\nu$ 
in \iref{series}. The error analysis is based on the available
$\ell^p$ summability for both sequences $(\|u_\nu\|_{V})_{\nu\in \cF}$
and $(\|u_\nu\|_{X})_{\nu\in \cF}$ where $X$ is a regularity
class that satisfies the spatial \emph{approximation property}
\be
\min_{v_n\in V_n} \|v-v_n\|_V \leq C_X n^{-t} \|v\|_{X}, \quad n > 0, \quad v\in X,
\label{spatialappn2}
\ee
for a family of $n$-dimensional subspaces $(V_n)_{n>0}$.   
We have focused our attention on the Hilbertian Sobolev spaces
$H^s$ for which \iref{spatialappn2} holds with $t=\frac {s-1}m$ when
using finite element spaces $V_n$ of sufficiently high
order on either uniformly refined meshes, or locally refined meshes in the
case of polygonal domains. 

We next would like to consider adaptive space discretizations,
such as adaptive finite elements or 
wavelet methods. In this setting, the space $V_n$ 
is no longer a linear space. For instance, in an adaptive 
wavelet method, it is described as the set of all possible $n$-term
combinations in the given wavelet basis $(\psi_\lambda)_{\lambda\in\cS}$, that is,
\be
\label{waveletapproxspaces}
V_n:=\Big\{\sum_{\lambda \in E} c_\lambda\psi_\lambda\; : \; (c_\lambda)_{\lambda\in E}\in\R^E, 
 E\subset \cS,\quad \#(E)\leq n\Big\}.
\ee
Therefore, when using such spaces, the resulting approximant
$u_{\bn}$ given in \iref{ntermh}, may be rewritten in the space-parameter basis 
$(\psi_\lambda \otimes \phi_\nu)_{(\lambda,\nu)\in\cS\times \cF}$ according to
\be
u_{\bn}= \sum_{\nu\in \Lambda_n}\sum_{\lambda\in E_\nu} c_{\lambda,\nu}\; \psi_\lambda \otimes \phi_\nu,
\ee
where $\#(E_\nu)=n_\nu$. When imposing the number of degrees of freedom
$N=\sum_{\nu\in \Lambda_n} n_{\nu}$, this means that 
$u_{\bn}$ can be picked as any $N$-term 
approximation
\be
u_{\bn}=\sum_{(\lambda,\nu)\in G_N} c_{\lambda,\nu} \;\psi_\lambda \otimes \phi_\nu,
\label{ntermadapt}
\ee
where $\#(G_N)=N$. The error of best $N$-term approximation is a natural
benchmark for space-parameter adaptive methods, as developed in \cite{BCD,G}.

One advantage of nonlinear families $(V_n)_{n\geq 1}$ such as given by
\iref{waveletapproxspaces} is that the smoothness 
conditions that govern a given approximation rate $n^{-t}$
are substantially weaker than with their linear counterparts.
Typical results, see e.g. \cite{Co,De}, are the following: if $D\subset \R^m$ is a bounded Lipschitz domain,
for $s>1$ and $t=\frac {s-1}m$ the approximation property
 \iref{spatialappn2} holds for the Sobolev spaces
\be
X=W^{s,\tau}(D), \quad   \frac   1 \tau < \min\Big\{1,\frac 1 2+\frac {s-1}m\Big\},
\ee
and more general Besov spaces
\be
X=B^{s,\tau}_q(D), \quad  \frac   1 \tau < \frac 1 2+\frac {s-1}m.
\ee
The limit case $\frac 1 \tau =\frac 1 2+\frac {s-1}m$ also holds if $q\leq \tau$.
These results hold provided that the 
degree of the chosen finite elements, or degree of polynomial reproduction
of the chosen wavelet systems, is larger than $s-1$.

Since these spaces are larger than $H^s$ when $\tau<2$,
we may hope that the summability index $p_X$ of $(\|u_\nu\|_{X})_{\nu\in \cF}$
is smaller when using such non-Hilbertian spaces for $X$, thereby leading to
improved rates of space-parameter approximation when using 
best $N$-term approximations of the form \iref{ntermadapt}.

In the following, we discuss this improved summability in a simple case.
We give a result for the Taylor and Jacobi coefficients in the case of affine parameter dependence
\iref{affine}. For $1\leq \tau\leq 2$, we introduce the Banach space 
\be
B^\tau:=\{v\in V\; : \; \Delta u\in L^\tau(D)\},
\ee
endowed with the norm and semi-norm
\be
\|v\|_{B^\tau}=\|v\|_V+ |v|_{B^\tau}, \quad |v|_{B^\tau}:=\|\Delta v\|_{L^\tau}.
\ee
Note that if $D$ is convex and $1<\tau\leq 2$,
we have by elliptic regularity that $B^\tau=W^{2,\tau}(D)$.

\begin{theorem}
\label{theotaylorwW2p}
Let $1\leq \tau<2$. 
Assume that $\bar a\in L^\infty(D)$ is such that $\essinf \bar a>0$,
and that there exists a sequence $\rho=(\rho_j)_{j\geq 1}$ of positive numbers
and a sequence $\bar \rho=(\bar \rho_j)_{j\geq 1}$ of numbers strictly larger than $1$,
such that 
\be
\left \| \frac{\sum_{j\geq 1} \bar \rho_j \rho_j|\psi_j|}{\bar a} \right \|_{L^\infty}=\theta <1,
\label{uearhodouble}
\ee 
and that $(\bar \rho_j^{-1})_{j\geq 1}\in \ell^{\bar q}(\N)$ where $\frac 1 2+\frac 1 {\bar q}=\frac 1 \tau$.
Assume in addition that the right side $f$ in \iref{ellip}
belongs to $H^{-1}(D)\cap L^\tau(D)$, and that $\bar a$ and all functions $\psi_j$
belong to $W^{1,\infty}(D)$ where
\be
\label{taylorwW2pweights}
\biggl\| \sum_{j\geq 1} \rho_j |\nabla\psi_j| \biggr\|_{L^\infty} <\infty.
\ee
Then
\be
\sum_{\nu\in\cF} (\rho^{\nu} \|t_\nu\|_{B^{\tau}})^\tau <\infty, \quad \rho^\nu:=\prod_{j\geq 1} \rho_j^{\nu_j}.
\label{taylorwW2p}
\ee
\end{theorem}

\begin{proof}
We take $\rho_j=1$, up to rescaling. We notice that when $a\in W^{1,\infty}(D)$, we can write the equation 
in the strong form
$- a \Delta u=\nabla a \cdot \nabla u+ f$, 
where all terms in the equality belong to $H^{-1}(D)$ and to $L^\tau(D)$. Hence
\be
- a(y) \Delta u(y)=\nabla a(y) \cdot \nabla u(y)+ f,
\ee
for all $y\in U$, and we can differentiate at $y=0$, which leads to the identities
\be
-\bar a \Delta t_\nu=\nabla \bar a\cdot \nabla t_\nu+\sum_{j\in {\rm supp}(\nu)} \nabla \psi_i\cdot \nabla t_{\nu-e_j}
+\sum_{j\in {\rm supp}(\nu)}\psi_j \Delta t_{\nu-e_j},
\ee
where all terms in the equality belong to $H^{-1}(D)$ and to $L^\tau(D)$. So pointwise
\be
| \Delta t_\nu|
\leq C \Big (|\nabla t_\nu|+\sum_{j\in {\rm supp}(\nu)} \kappa_j |\nabla t_{\nu-e_j}|\Big )
+\theta \sum_{j\in {\rm supp}(\nu)}\omega_j|\Delta t_{\nu-e_j}|,
\ee
where $\kappa_j:= \frac{|\nabla \psi_i|}{\sum_{j\geq 1}|\nabla \psi_i|}$ and 
$\omega_j:= \frac{|\psi_i|}{\sum_{j\geq 1}|\psi_i|}$ so that 
$\sum_{j\geq 1} \kappa_j=\sum_{j\geq 1} \omega_j= 1$, and where $C>1$ is a fixed constant.
We elevate to the power $\tau$ and use the observation that 
for any $\e>0$, there exists a constant $C=C(\e,\tau)>1$ such that for any $a,b\geq 0$,
\be
(a+b)^\tau \leq  C a^\tau + (1+\e)b^\tau.
\ee
Taking $\e$ small enough, so that $\bar \theta:=(1+\e)\theta^\tau<1$, and using the convexity of $x\mapsto |x|^\tau$, 
we obtain
\be
| \Delta t_\nu|^\tau
\leq C \Big (|\nabla t_\nu|^\tau +\sum_{j\in {\rm supp}(\nu)} \kappa_j |\nabla t_{\nu-e_j}|^\tau\Big )
+\bar\theta \sum_{j\in {\rm supp}(\nu)}\omega_j|\Delta t_{\nu-e_j}|^\tau,
\ee
for some fixed $C>1$. Integrating and summing over $|\nu|=n$ thus gives
\be
\sum_{|\nu|=n}\|\Delta t_\nu\|_{L^\tau}^\tau \leq  C \Big (\sum_{|\nu|=n}\|\nabla t_\nu\|_{L^\tau}^\tau+
\sum_{|\nu|=n-1}\|\nabla t_\nu\|_{L^\tau}^\tau\Big )+\bar \theta \sum_{|\nu|=n-1}\|\Delta t_\nu\|_{L^\tau}^\tau.
\label{eqtau}
\ee
Since we have
\be
\left \| \frac{\sum_{j\geq 1} \bar \rho_j |\psi_j|}{\bar a} \right \|_{L^\infty}=\theta <1,
\label{uearhosimple}
\ee 
with $(\bar \rho_j^{-1})_{j\geq 1}\in \ell^{\bar q}(\N)$ where $\frac 1 2+\frac 1 {\bar q}=\frac 1 \tau$,
application of Theorem \ref{theotaylorV} gives us that $(\|t_\nu\|_V)_{\nu\in\cF}\in \ell^\tau(\cF)$.
Since we assumed $\tau<2$,
this implies that $(\|\nabla t_\nu\|_{L^\tau})_{\nu\in\cF}\in \ell^\tau(\cF)$.
Since $\bar \theta <1$, we conclude by summing \iref{eqtau} over $n$ that
\[
\sum_{\nu\in\cF}\|\Delta t_\nu\|_{L^\tau}^\tau<\infty,
\]
which implies the $\ell^\tau$ summability of $(\| t_\nu\|_{B^\tau})_{\nu\in\cF}$.
\end{proof}

\begin{remark}\label{remjacobiweightedl2}
By proceeding in a similar way as in the proof of Theorem \ref{jacobiweightedl2},
we may extend the above result to Jacobi coefficients under the same assumptions: one has 
\begin{equation}\label{jacobiwW2p}
\sum_{\nu\in\cF} (a_\nu^{-1} \rho^\nu \|v_\nu\|_{B^\tau} )^\tau < \infty
\end{equation}
with $a_\nu$ as in Theorem \ref{jacobiweightedl2}. 
\end{remark}

If $\rho_j >1$ and $(\rho_j)_{j\geq 1} \in \ell^q(\N)$, by H\"older's inequality we obtain 
\be
\label{holdertau}
   \sum_{\nu\in\cF} \|t_\nu\|_{B^\tau}^p 
   \leq \Bigl( \sum_{\nu\in\cF} \bigl( \rho^\nu \norm{ t_\nu }_{B^\tau} \bigr)^\tau  \Bigr)^{\frac{q}{\tau+q}} \Bigl( \sum_{\nu\in\cF}  \rho^{-q\nu}\Bigr)^{\frac{\tau}{q+\tau}}, \quad \frac1p = \frac1q + \frac1\tau.
\ee
This needs to be compared to \eqref{holder}, which based on the same condition \eqref{taylorwW2pweights} constraining $q$ only gives $\frac1p = \frac1q + \frac12$.
Concerning the spatial approximation rate $t$, we observe the following: for space dimension $m=1$, we directly obtain the rate $t = \frac1m = 1$ for nonlinear approximation in $V$ of elements of $B^\tau = W^{2,\tau}(D)$ with $\tau\geq 1$; for $m=2$, we still have $B^\tau = W^{2,\tau}(D)$ for $\tau > 1$ using the above elliptic regularity result, which then again gives $t=\frac1m = \frac12$; whereas for $m=3$, we obtain $t = \frac1m = \frac13$ under the stronger condition $\tau \geq \frac65$. In summary, nonlinear approximation in space gives us the same spatial approximation rate with better summability of the corresponding higher-order norms.

We next use interpolation, similarly to Theorem \ref{theotaylorfrac}, to extend the above results to cases where \eqref{taylorwW2pweights} is not satisfied for any $\rho_j>1$.  

\begin{cor}
\label{cortaylorwW2pinterp}
Let $1\leq \tau<2$. 
Assume that $\bar a\in L^\infty(D)$  is such that $\essinf \bar a>0$, that $f$ in \iref{ellip}
belongs to $H^{-1}(D)\cap L^\tau(D)$ and that $\bar a$ and all functions $\psi_j$
belong to $W^{1,\infty}(D)$. In addition, assume that there exist sequences $(\rho_j)_{j\geq 1}$, $(\hat \rho_j)_{j\geq 1}$ of positive numbers such that 
\be\label{cortaylorwW2pinterpcond}
 \left \| \frac{\sum_{j\geq 1} \hat \rho_j |\psi_j|}{\bar a} \right \|_{L^\infty} <1, \quad 
 \biggl\| \sum_{j\geq 1} \rho_j |\nabla\psi_j| \biggr\|_{L^\infty} <\infty,
\ee 
where $\hat\rho_j / \rho_j > 1$ and $( \rho_j / \hat\rho_j )_{j\geq 1}\in \ell^{q}(\N)$ with $\frac 1 2+\frac 1 {q}=\frac 1 \tau$.
Then with $\tilde\rho_j := \hat\rho_j^{1-\theta} \rho_j^\theta$ and $\tilde \rho^\nu:=\prod_{j\geq 1}  \tilde \rho_j^{\nu_j}$, for all $0<\theta<1$ and $0<r\leq \infty$,
\be\label{taylorwW2pinterp}
\sum_{\nu\in\cF} \bigl(\tilde\rho^{\nu} \norm{t_\nu}_Z \bigr)^\zeta <\infty\quad\text{and}\quad \sum_{\nu\in\cF} \bigl(a_\nu^{-1}\tilde\rho^{\nu} \norm{v_\nu}_Z \bigr)^\zeta<\infty, \qquad \frac1\zeta = \frac12 + \biggl( \frac1\tau - \frac12 \biggr) \theta,
\ee
where $Z=[V,B^\tau]_{\theta,r}$ or $Z=[V,B^\tau]_{\theta}$, and where $a_\nu$ is defined in Theorem \ref{jacobiweightedl2}.
\end{cor}

\begin{proof}
From the interpolation inequality \iref{interineq}, there exists $C>0$ such that 
\[
\norm{t_\nu}_Z \leq C \norm{t_\nu}_V^{1-\theta} \norm{t_\nu}^\theta_{B^\tau}\]
 for all $\nu\in\cF$. By H\"older's inequality,
\[
\begin{aligned}
	\Bigl( \sum_{\nu\in\cF} \bigl(\tilde\rho^\nu \norm{t_\nu}_Z)\bigr)^{\zeta} \Bigr)^{\frac1{\zeta}} &
	  \leq C \Bigl( \sum_{\nu\in\cF} \bigl( \hat \rho^\nu \norm{t_\nu}_V\bigr)^{(1-\theta) \zeta} \bigl( \rho^\nu \norm{t_\nu}_{B^\tau}\bigr)^{\theta \zeta} \Bigr)^{\frac1{\zeta}}  \\
	  &\leq C \Bigl( \sum_{\nu\in\cF}  \bigl( \hat \rho^\nu \norm{t_\nu}_V\bigr)^{2}  \Bigr)^{ \frac{1-\theta}{2} } \Bigl(  \sum_{\nu\in\cF} \bigl( \rho^\nu \norm{t_\nu}_{B^\tau}\bigr)^{\tau} \Bigr)^{ \frac\theta\tau },
\end{aligned}
\]
and the right-hand side is finite by Theorems \ref{theotaylorV} and \ref{theotaylorwW2p}. The analogous statement for the Jacobi coefficients follows with Remark \ref{remjacobiweightedl2}.
\end{proof}

We now verify that the interpolation spaces considered above indeed have the expected approximation properties.

\begin{prop}\label{propinterpapprox}
Let $0<\theta<1$ and $0<r\leq \infty$. 
Consider a nonlinear family $(V_n)_{n\geq 1}$ which satisfies the approximation property
\iref{spatialappn2} for $X=B^\tau$ and $t=\frac 1 m$, for some $\tau\geq 1$.
Then, for any $0<\theta<1$ and $0<r\leq \infty$, the approximation property
\be
\min_{v_n\in V_n} \|v-v_n\|_V \leq C_Z \, n^{-\theta/m} \|v\|_{Z}, \quad n > 0,
\label{spatialappn3}
\ee
holds for $Z=[V,B^\tau]_{\theta,r}$ or $Z=[V,B^\tau]_{\theta}$.
\end{prop}

\begin{proof} 
Since for any admissible pair $(V,X)$ of Banach spaces, one has
\be
[V,X]_{\theta,1}\subset [V,X]_{\theta}\subset [V,X]_{\theta,\infty},
\ee
it is sufficient to consider $Z=[V,B^\tau]_{\theta,\infty}$. 
Now for any $v\in V$ and $w\in X$, we may write
\be
\min_{v_n\in V_n} \|v-v_n\|_V \leq \|v-w\|_V+\min_{v_n\in V_n} \|w-v_n\|_V
\leq \|v-w\|_V+C_X n^{-1/m} \|w\|_{X}.
\ee
Therefore
\be
\min_{v_n\in V_n} \|v-v_n\|_V \leq \min_{w\in X}\{\|v-w\|_V+C_X n^{-1/m} \|w\|_{X}\}.
\ee
The right-hand side is the $K$-functional $K(v,C_X n^{-1/m},V,X)$, which by definition
of interpolation spaces satisfies
\be
K(v,C_X n^{-1/m},V,X)\leq C_X^\theta n^{-\theta/m} \|v\|_{Z}
\ee
 when $v\in Z$, thereby proving \iref{spatialappn3} with $C_Z:=C_X^\theta$.
\end{proof}
\section{Multiresolution representation of $a(y)$}
\label{sec:MRA}
We finally illustrate our results in the particular 
case of an affine wavelet-type parametrization of the coefficient $a(y)$. 
Our focus will be on Jacobi expansions in $L^2(U,V,\sigma)$. Concerning analogous results for 
Taylor approximation in $L^\infty(U,V)$, see Remark \ref{wavLinfty}. 
We refer to \cite{Co} for a general treatment of
wavelets and their adaptation to a bounded domain $D\subset \R^m$, and
summarize below the main properties that are needed in our analysis.

\begin{ass}\label{asswavelets}
Let $(\psi_\lambda)_{\lambda\in\cS}$ be a family of wavelet basis functions, where the scale-space indices $\lambda$ comprise dilation and translation parameters, with the convention that the scale $l$
of $\psi_\lambda$ is denoted by $|\lambda|:= l$, and where the number of wavelets on level $\abs{\lambda}=l$ is proportional to $2^{ml}$. 
Moreover, we assume the wavelets at each given scale to have finite overlap, 
that is, there exists $M>0$ such that for all $x\in D$ and $l$,
\begin{equation}
\label{finiteoverlap}
\#\{\lambda \, : \, \abs{\lambda}=l \; {\rm and} \;\psi_\lambda(x)\neq 0\} \leq M.
\end{equation}
For simplicity, we take $\bar a = 1$, assume $D$ and $f$ to be smooth, 
and fix an ordering $(\lambda(j))_{j\geq 1}$ of the indices
from coarser to finer scales, for which we set $\psi_j := \psi_{\lambda(j)}$. 
We consider wavelets with $\psi_\lambda \in W^{\kappa,\infty}(D)$ for some $\kappa \in \N$, 
that are normalized such that
\be
\|\psi_\lambda\|_{L^\infty} =C 2^{-\alpha \abs{\lambda}}, 
\ee
for an $\alpha \in (0, \kappa)$ and $C>0$ chosen such that \eqref{uea} holds. The
partial derivatives behave like
\be
\|D^\mu\psi_\lambda\|_{L^\infty} \sim 2^{-(\alpha-|\mu|) \abs{\lambda}},
\ee
for $|\mu| \leq \kappa$.
\end{ass}

For ease of exposition, we assume in what follows that $\alpha$ is not an integer.
Under the above assumptions, for $0< \beta < \alpha$ we have 
\be
\sup_{x\in D}\sum_{\lambda} 2^{\beta |\lambda|} |\psi_\lambda(x)| <\infty
\ee
as a consequence of \eqref{finiteoverlap}.
Therefore, with 
\be\label{waveletrhoj}
 \rho_j := 1 + c 2^{\beta \abs{\lambda(j)}} \sim j^{\beta/m}
\ee
with $c>0$ is sufficiently small, we have \eqref{uearho}, and thus obtain the following from Theorem \ref{theotaylorV}.

\begin{prop}
\label{waveletV}
Under Assumptions \ref{asswavelets} one has $(\|v_\nu\|_V)_{\nu\in \cF} \in \ell^p(\cF)$ for any $p > (\frac\alpha{m} + \frac12)^{-1}$, and consequently
\begin{equation}\label{wavnterm}
\Bignorm{ u- \sum_{\nu\in\Lambda_n} v_\nu J_\nu }_{L^2(U,V,\sigma)} \leq  C \,n^{-s}, 
\end{equation}
holds for any $s < \frac\alpha{m}$. Here, $C:=\|(\|v_\nu\|_V)_{\nu\in \cF}\|_{\ell^p}$ 
with $s=\frac 1 p-\frac 1 2$, and $\Lambda_n$
is the set of indices $\nu$ corresponding to the $n$ largest $\|v_\nu\|_V$.
\end{prop}

\subsection{Sobolev regularity}\label{sec:sobolevregexample}

In order to make use of our results for higher-order spatial regularity, we observe that for any $\bar\beta < \alpha - (k-1)$ and 
$\bar\rho_j := 2^{\bar\beta \abs{\lambda(j)}} \sim j^{\bar\beta/m}$, also
\begin{equation}
 \sup_{|\mu|\leq k-1} \Bignorm{  \sum_{j\geq 1} \bar\rho_j \abs{D^\mu \psi_j(x)}  }_{L^\infty} < \infty.
\end{equation}
Based on Theorems \ref{theotaylorHm} and \ref{theojacobiHm}, which concern integer-order Sobolev regularity, and on Theorem \ref{theoconvrate}, we obtain the following conclusions for the fully discrete approximations.

\begin{cor}\label{corwaveletHk}
	Let Assumptions \ref{asswavelets} hold with $\alpha>1$. Let $k\in\{2,\ldots,\ceil{\alpha}\}$, and let the spatial approximation spaces $(V_n)_{n>0}$ satisfy the approximation property
	\begin{equation} \label{wvspapp}
 \min_{w_n\in V_n} \|w-w_n\|_V \leq C_{k} n^{-t}\|w\|_{H^k},  
	 \end{equation}
with $t = \frac{k-1}m$.
Then for any $r < \frac\alpha{2m}$, there exists a constant $C>0$ such that
the following holds: for each $n$ there exists $\bn = (n_\nu)_{\nu\in\Lambda_n}$  
such that
\be\label{waveletconvHk}
	\min_{u_{\bn}\in V_{\bn}} \|u-u_{\bn}\|_{L^2(U,V,\sigma)} \leq C N^{-\min\{r,t\}},
\ee
where $N := \sum_{\nu\in\Lambda_n} n_\nu={\rm dim}(V_{\bn})$.
\end{cor}

\begin{proof}
	Theorems \ref{theotaylorHm} and \ref{theojacobiHm} can be applied when $k - 1 < \alpha$ with integer $k\geq 2$. Since by assumption $\alpha$ is not an integer, the largest value for $k$ that we can use is thus $k = \ceil{\alpha}$.
	Then Theorems \ref{theotaylorHm} and \ref{theojacobiHm} yield $(\norm{v_\nu}_{H^k})_{\nu\in \cF}\in\ell^{\bar p}(\cF)$ for $\bar p > (\frac{\bar\beta}m + \frac12)^{-1} > (\frac{\alpha - k + 1}m + \frac12)^{-1}$. 
By Theorem \ref{theoconvrate} with $\cV = L^2(U,V,\sigma)$, we obtain the overall convergence rate $\min\{r,t\}$, where
\be
\label{rateL2}
   r = \frac{st}{t + s + \frac12 -\bar p^{-1}}.
\ee
Taking $s\to \frac\alpha{m}$ and $\bar p \to (\frac{\alpha - k + 1}m + \frac12)^{-1}$,  we obtain $r \to \frac\alpha{2m}$.
\end{proof}

\begin{remark}\label{rmrkwaveletconvHk}
Provided that the approximation property in \eqref{wvspapp} holds for sufficiently large $k$, one has $\min\{r,t\} = r$. For instance, as soon as $1 + \frac12 \alpha \leq k \leq \ceil{\alpha}$ (where $\alpha>1$), the conclusion in the above corollary reads as follows:
for any $\delta >0$, there exists a constant $C>0$ such that,
for each $n$, there exists $\bn = (n_\nu)_{\nu\in\Lambda_n}$  
such that
\be\label{waveletconvHk2}
	\min_{u_{\bn}\in V_{\bn}} \|u-u_{\bn}\|_{L^2(U,V,\sigma)} \leq C N^{-\frac \alpha{2m}+\delta},
\ee
where $N := \sum_{\nu\in\Lambda_n} n_\nu={\rm dim}(V_{\bn})$.
Conversely, one has $\min\{ r, t\} = t$ in \eqref{waveletconvHk} only when an artificially low value of $k$ is used. This corresponds to the situation when one is limited to a certain order of spatial approximation $t$ that is lower than what could in principle be exploited for the given $\alpha$.
\end{remark}

We now show that the restriction to $\alpha>1$ in Corollary \ref{corwaveletHk} can be removed using the interpolation argument from Corollary \ref{theotaylorfrac}. Under Assumptions \ref{asswavelets}, the hypotheses of this corollary are satisfied for $k := \ceil{\alpha} + 1$, and by our smoothness assumption on $D$, for $0<\theta<1$ we have
\[
 Z:=[V, W^k]_{\theta,2} = [V, V\cap H^k(D)]_{\theta,2} = V\cap H^{1+\theta (k-1)}(D).
\]
On the one hand, we take $\rho_j = 1 + c 2^{\beta \abs{\lambda(j)}}$ as in \eqref{waveletrhoj} with positive $\beta < \alpha$ and sufficiently small $c>0$ to ensure \eqref{uearho}.
On the other hand, for $\bar\beta < \alpha - \ceil{\alpha} < 0$, taking $\bar\rho_j :=  c^{(\theta-1)/\theta} 2^{\bar\beta\abs{\lambda(j)}}$ we also have
\begin{equation}
  \Bignorm{  \sum_{j\geq 1} \bar\rho_j \abs{D^\mu \psi_j(x)}  }_{L^\infty} < \infty, \quad \abs{\mu} \leq k-1 = \ceil{\alpha}.
\end{equation}
Applying Corollary \ref{theotaylorfrac} and Remark \ref{fracjacobi}, we obtain $\tilde\rho_j > 1$ in \eqref{fracjacobibound} provided that $\theta$ is chosen to satisfy $(1-\theta)\beta + \theta\bar\beta >0$, which is ensured to be achievable provided that 
\be
\label{thetabound}
0 < \theta < \alpha / \ceil{\alpha} .
\ee
With suitably chosen $\beta$, $\bar\beta$ we then obtain $(\norm{v_\nu}_Z)_{\nu\in\cF} \in \ell^{\bar p}(\cF)$ for any $\bar p$ such that
\[
  \bar p^{-1} < \frac{\alpha - \theta\ceil{\alpha}}{m}  + \frac12.
\]
In Theorem \ref{theoconvrate}, for elements of $Z$ we can achieve the spatial rate $t = \theta\ceil{\alpha}/m$. 
As $\theta \to \alpha / \ceil{\alpha}$, we thus obtain the following conclusion, extending Corollary \ref{corwaveletHk} to $\alpha \in ]0,1[$.

\begin{cor}\label{corwaveletHkinterp}
Let \eqref{wvspapp} hold with $k\geq 2$, and let Assumptions \ref{asswavelets} hold with $0 < \alpha <1$.
Then, for any $\delta >0$, there exists a constant $C>0$ such that
the following holds: for each $n$ there exists $\bn = (n_\nu)_{\nu\in\Lambda_n}$  such that
\begin{equation}
	\min_{u_{\bn}\in V_{\bn}} \|u-u_{\bn}\|_{L^2(U,V,\sigma)} \leq C N^{-\frac\alpha{2m}+\delta}
\end{equation}
where $N := \sum_{\nu\in\Lambda_n} n_\nu={\rm dim}(V_{\bn})$.

\end{cor}

\begin{remark}\label{wavLinfty}
	For obtaining rates in $L^\infty(U,V)$, we need to additionally require $\alpha>\frac{m}2$. One then has \eqref{wavnterm} for any $s < \frac\alpha{m} - \frac12$, and analogously to Corollaries \ref{corwaveletHk} and \ref{corwaveletHkinterp}, one obtains \eqref{waveletconvHk} also for Taylor approximation with any $r < \frac12 \left( \frac\alpha{m} - \frac12  \right)$.
\end{remark}

\begin{remark}
	The limiting rate $\frac{\alpha}{2m}$ for approximation in $L^2(U,V,\sigma)$, which we obtain here using optimized $\nu$-dependent spatial discretizations, can also be achieved by a simpler construction using a single $\nu$-independent spatial discretization space. This can be seen as follows: Under our present assumptions, $\sup_{y\in U} \norm{a(y)}_{C^\alpha} < \infty$ and consequently $\sup_{y\in U} \norm{u(y)}_{H^{1+\beta}} < \infty$ for any positive $\beta< \alpha$. Hence for the Jacobi coefficients, we have
	\[
	    \sum_{\nu\in\cF} \norm{v_\nu}_{H^{1+\beta}}^2 = \int_{U} \norm{ u(y) }_{H^{1+\beta}}^2\,d\sigma(y) < \infty.
	\]
	For $n\in\N$, as in \S\ref{sec:SpcParDiscr} let $\Lambda_n\subset \cF$ 
        be a subset of $n$ largest $\norm{v_\nu}_{V}$. 
        Now choose the vector $\bn$ as in \eqref{bndof} as $n_\nu = \hat n$ 
        with some fixed $\hat n$ for all $\nu\in\Lambda_n$.
	With an appropriate sequence of spatial approximation spaces $(V_n)_{n>0}$, we then obtain
	\begin{equation*}
	   \norm{ u - u_\bn }_\cV \leq \Bignorm{ \sum_{\nu\in\cF} (v_\nu - v_{\nu,\hat n}) J_\nu }_\cV + \Bignorm{ \sum_{\nu \notin \Lambda_n} v_{\nu,\hat n} J_\nu }_\cV 
	     \leq C \hat n^{-t} \Bigl( \sum_{\nu\in\cF} \norm{v_\nu}_{H^{1+\beta}}^2 \Bigr)^{\frac12} + C n^{-s},
	\end{equation*}
	with $t = \frac{\beta}m < \frac{\alpha}m$ and any $s < \frac{\alpha}m$. 
        In this case, the total number of coefficients is simply $N = n \hat n$, 
        which means that with appropriately chosen $n$, $\hat n$ we obtain
	\begin{equation}
		\norm{ u - u_\bn }_\cV \leq C N^{-\frac\alpha{2m} + \delta}
	\end{equation}
	for any $\delta>0$, and hence the same asymptotic convergence as in Corollaries \ref{corwaveletHk} and \ref{corwaveletHkinterp}. In other words, in the present setting, individual optimization of discretization spaces for each Jacobi coefficient does not improve the achievable convergence rate compared to a single spatial discretization for all coefficients.
\end{remark}

In summary, under Assumptions \ref{asswavelets}, 
provided that spatial approximation by linear methods of sufficiently high order is used (e.g., based on a predefined hierarchy of finite element meshes), 
the limiting spatial-parametric rate is \emph{half} the limiting convergence 
rate of the parametric expansion.

\subsection{Space-parameter adaptivity}\label{sec:adaptexample}

We next consider the convergence rates resulting from a combination of Theorem \ref{theoconvrate} with Theorem \ref{theotaylorwW2p} and Corollary \ref{cortaylorwW2pinterp}, under the particular Assumptions \ref{asswavelets} on a wavelet-type expansion of the coefficient $a$. To simplify the exposition, we focus here on Jacobi expansions in $L^2(U,V,\sigma)$ with spatial dimensions $m=2,3$.

\begin{cor}\label{corwavnonlin}
Let Assumptions \ref{asswavelets} hold, and let $\tau > 1$ if $m=2$ and $\tau \geq \frac65$ if $m=3$. Assume that the approximation property \iref{spatialappn2}
holds for a nonlinear family $(V_n)_{n\geq 1}$ such as \iref{waveletapproxspaces},
with $X = B^\tau$ and $t=\frac1m$ when $\alpha > 1$, or otherwise
that the approximation property \iref{spatialappn3} holds with $Z = [V,B^\tau]_\theta$
or $ [V,B^\tau]_{\theta,r}$ for $0<\theta<\alpha < 1$ and $0<r\leq\infty$.
Then, for any $\delta>0$ and $\alpha>0$,  there exists a constant $C>0$ such that
the following holds: for each $n$ there exists $\bn = (n_\nu)_{\nu\in\Lambda_n}$,
such that, when $0< \alpha <1$, 
\begin{equation}\label{wvnonlin1}
	 \norm{ u - u_\bn }_{L^2(U,V,\sigma)} \leq C N^{-\frac\alpha{m} + \delta},
\end{equation}
and when $\alpha>1$,
\begin{equation}\label{wvnonlin2}
	 \norm{ u - u_\bn }_{L^2(U,V,\sigma)} \leq C N^{-\frac1{m} },
\end{equation}
where $N:=\sum_{\nu\in\Lambda_n} n_\nu$.
\end{cor}

\begin{proof}
	In the case $\alpha > 1$, we combine Theorem \ref{theotaylorwW2p} with the bound \eqref{holdertau} to obtain $(\norm{t_\nu}_{B^\tau})_{\nu\in\cF} \in \ell^p$ for any $p$ such that $\frac1p < \frac{\alpha-1}{m} + \frac1\tau$. By Proposition \ref{waveletV}, we have convergence of the parametric expansion with any $s < \frac\alpha{m}$.
	We now again apply Theorem \ref{theoconvrate}. As $s$, $p$, and $\tau$ approach their limiting values,
	\be\label{nonlinratelimit}
	  \frac{st}{s+t+\frac12 - p^{-1}} \to \frac{\alpha}m,
	\ee
	resulting in the convergence rate $\min\{\frac{\alpha}m,t\} = \frac1m$.
	
	In the case $0<\alpha < 1$, we apply Corollary \ref{cortaylorwW2pinterp} with $0<\theta<\alpha$ and weight sequences $(\rho_j)_{j\geq 1}$ and $(\hat\rho_j)_{j\geq 1}$ constructed as follows: We can choose $\beta\in]-1,0[$ such that $\beta < \alpha - 1$ and $\hat\beta>0$ such that $\beta+1 < \hat\beta < \alpha$ satisfying $(1-\theta)\hat\beta + \theta \beta>0$. We set $\hat\rho_j := 1+c 2^{\hat\beta \abs{\lambda(j)}}$ with $c>0$ sufficiently small to ensure the first inequality in \eqref{cortaylorwW2pinterpcond}. With $J$ the smallest integer such that $c^{1-\theta} 2^{((1-\theta)\hat\beta + \theta\beta) \abs{\lambda(j)}} > 1$, let $\rho_j := 1$ for $j < J$ and $\rho_j := 2^{\beta\abs{\lambda(j)}}$ for $j\geq J$. With these choices, the assumptions of Corollary \ref{cortaylorwW2pinterp} are satisfied and $\tilde \rho_j > 1$. By taking $\beta$ and $\hat\beta$ sufficiently close to $\alpha-1$ and $\alpha$, respectively, we obtain $(\tilde\rho_j^{-1})_{j\geq 1} \in \ell^q(\N)$ for $\frac1q < \frac{\alpha - \theta}m$. By \eqref{holdertau}, with $\zeta$ as in \eqref{taylorwW2pinterp}, we have $(\norm{t_\nu}_Z)_{\nu\in\cF} \in \ell^p(\cF)$ for any $p$ such that
\[
   \frac1p < \frac{\alpha - \theta}m + \frac1\zeta = \frac\alpha{m} + \frac12 +\left( \frac1\tau -\frac12 - \frac1m \right) \theta.
\]
By Proposition \ref{propinterpapprox}, under our conditions on $\tau$, we have $t = \frac\theta{m}$ for elements of $Z = [V,B^\tau]_\theta$ or $[V,B^\tau]_{\theta,r}$.
As $s$, $p$, $\tau$ approach their respective limiting values we again obtain \eqref{nonlinratelimit}, independently of $\theta$.
 For the resulting limiting convergence rate, we thus obtain $\min\{\frac{\alpha}m , \frac{\theta}m \} \to \frac{\alpha}m$ as $\theta\to\alpha$.
\end{proof}

\begin{remark}\label{wavm1}
Following the same lines for $m=1$, we instead obtain the rate $\min\{ \frac{2\alpha}{3} - \delta, 1\}$. Improving this to $\min\{\alpha-\delta,1\}$, analogously to our results for $m>1$, would require summability results for appropriate Besov norms with $\tau = \frac23$. The results of numerical experiments given in \cite[Fig.\ 3(a)]{BCD} for the present setting with $m=1$, however, indicate that the rate $\min\{ \frac{2\alpha}{3} - \delta, 1\}$ is in general sharp. For further numerical illustrations, we refer to \cite[\S5.3]{BCM} and \cite[\S6.2]{BCD}.
\end{remark}

In summary, in contrast to the case of linear approximation in the spatial variable, for $0<\alpha <1$ we arrive at the same rate as for the separate approximations in spatial or parametric variables alone. 
For $\alpha > 1$, the convergence rates obtained here are limited to $\frac1m$, since 
the results obtained in \S 7 do not cover the $\ell^p$ summability of 
the sequence $(\|u_\nu\|_X)_{\nu\in \cF}$ for the higher-order smoothness classes
$X=W^{s,\tau}(D)$ or $B^{s,\tau}_q(D)$ when $s>2$.

\begin{remark}
In \cite{BCD}, it is shown for Legendre expansions that when a suitable wavelet basis is used for the spatial discretization, under Assumption \ref{asswavelets} with sufficiently regular $\psi_j$ one can achieve $s^*$-compressibility of the resulting representation of the parametric diffusion operator on $L^2(U,V,\sigma)$ with $s^*$ arbitrarily close to $\frac\alpha{m}$. This implies that approximations $u_\bn$ as in Corollary \ref{corwavnonlin} can be computed by standard adaptive methods with the number of operations scaling almost linearly in the total number of degrees of freedom $N$ for any $\alpha>0$.
\end{remark}

\end{document}